\documentclass[a4paper,11pt]{amsart}
\usepackage[margin=1in]{geometry}
\usepackage{amssymb,amsmath, amsthm}
\usepackage{mathrsfs}
\usepackage{stmaryrd}
\usepackage{tikz-cd}
\usetikzlibrary{matrix,arrows,decorations.pathmorphing}
\usepackage{hyperref}
\hypersetup{breaklinks, colorlinks}

\newtheorem{theorem}{Theorem}[section]
\newtheorem{lemma}[theorem]{Lemma}
\newtheorem{cor}[theorem]{Corollary}
\newtheorem{prop}[theorem]{Proposition}

\theoremstyle{definition}
\newtheorem{defn}[theorem]{Definition}
\newtheorem{ex}[theorem]{Example}
\newtheorem{remark}[theorem]{Remark}
\newtheorem{conj}{Conjecture}

\numberwithin{equation}{section}

\def\D{\mathbf{D}^\mathrm{b}}
\def\DD{\mathbf{D}}

\def\A{\mathcal{A}}
\def\B{\mathcal{B}}
\def\E{\mathcal{E}}
\def\F{\mathcal{F}}
\def\G{\mathcal{G}}
\def\I{\mathcal{I}}
\def\J{\mathcal{J}}
\def\L{\mathcal{L}}
\def\N{\mathcal{N}}
\def\O{\mathcal{O}}
\def\Q{\mathcal{Q}}
\def\R{\mathcal{R}}

\def\cC{\mathcal{C}}
\def\cH{\mathcal{H}}
\def\cK{\mathcal{K}}
\def\cS{\mathcal{S}}

\def\cV{\mathcal{V}}
\def\cW{\mathcal{W}}

\def\fS{\mathfrak{S}}

\def\P{\mathbb{P}}

\def\AA{\mathbb{A}}
\def\LL{\mathbb{L}}
\def\RR{\mathbb{R}}
\def\ZZ{\mathbb{Z}}

\newcommand{\kk}{\Bbbk}
\newcommand{\wt}{\widetilde}
\newcommand{\ts}{\tilde{\Sigma}}
\newcommand{\bB}{\bar{\B}}

\newcommand{\cchar}{\operatorname{char}}
\newcommand{\IIm}{\operatorname{Im}}
\newcommand{\rank}{\operatorname{rank}}
\newcommand{\tr}{\operatorname{tr}}
\newcommand{\Pic}{\operatorname{Pic}}
\newcommand{\Bl}{\operatorname{Bl}}
\newcommand{\Gr}{\operatorname{Gr}}
\newcommand{\Proj}{\operatorname{Proj}}
\newcommand{\Spec}{\operatorname{Spec}}
\newcommand{\Coh}{\operatorname{Coh}}
\newcommand{\QCoh}{\operatorname{QCoh}}
\newcommand{\perf}{\operatorname{perf}}
\newcommand{\Mod}{\operatorname{Mod}}
\newcommand{\op}{\operatorname{op}}

\newcommand{\Ext}{\operatorname{Ext}}
\newcommand{\Hom}{\operatorname{Hom}}
\newcommand{\coker}{\operatorname{coker}}

\newcommand{\Sym}{\operatorname{Sym}}
\newcommand{\id}{\operatorname{id}}
\newcommand{\Tot}{\operatorname{Tot}}
\newcommand{\Br}{\operatorname{Br}}
\newcommand{\ev}{\operatorname{ev}}

\newcommand{\sEnd}{\mathscr{E}\kern -1pt nd}
\newcommand{\sHom}{\mathscr{H}\kern -2pt om}

\title{Residual categories of quadric surface bundles}
\author{Fei Xie}
\address{School of Mathematics, University of Edinburgh, James Clerk Maxwell Building, Peter Guthrie Tait Road, Edinburgh, EH9 3FD, UK}
\email{fei.xie@ed.ac.uk}

\begin{document}
\begin{abstract}
We show that the residual categories of quadric surface bundles are equivalent to the (twisted) derived categories of some scheme under the following hypotheses. Case 1: The quadric surface bundle has a smooth section. Case 2: The total space of the quadric surface bundle is smooth and the base is a smooth surface. We provide two proofs in Case 1 describing the scheme as the hyperbolic reduction and as a subscheme of the relative Hilbert scheme of lines, respectively. In Case 2, the twisted scheme is obtained by performing birational transformations to the relative Hilbert scheme of lines. Finally, we apply the results to certain complete intersections of quadrics.
\end{abstract}

\maketitle

\tableofcontents

\section{Introduction}
For a flat family $f\colon X\to S$ of Fano varieties of index $n$, there is a semiorthogonal decomposition (SOD)
\begin{equation*}
    \D(X) = \langle \A_X, f^*\D(S) \otimes \O_X(1), \dots, f^*\D(S) \otimes \O_X(n) \rangle
\end{equation*}
where $\A_X$ is called the \textit{residual category} of $X$. In other words, the residual category $\A_X$ is the non-trivial component in the derived category. We can view $\A_X$ as a refined invariant of $X$. For example, the refined derived Torelli problem asks if $X$ is determined by $\A_X$. It is interesting to see when $\A_X$ is (twisted) geometric, i.e., equivalent to the (twisted) derived category of some scheme. In this paper, we study the problem when $f$ is a flat quadric surface bundle allowing fibers of corank $2$ and prove that their residual categories are (twisted) geometric in two cases.

Let $p\colon \Q\to S$ be a flat quadric surface bundle where $S$ is an integral noetherian scheme over a field $\kk$ with $\cchar(\kk)\neq 2$. Let $\A_{\Q}$ be the residual category of $p$. We start with Case 1. Assume that $p$ has a smooth section as in Definition \ref{regisodef} and the second degeneration $S_2\subset S$ is different from $S$ (The $k$-th degeneration $S_k$ is the locus in $S$ where fibers of $p$ have corank at least $k$.). In this case, $\A_{\Q}$ is geometric.
{\renewcommand{\thetheorem}{\ref{main1}}
\begin{theorem}
In the hypotheses of Case 1, $\A_{\Q}\cong\D(\bar{\Q})$ where $\bar{\Q}$ is the hyperbolic reduction of $\Q$ with respect to the smooth section in Definition \ref{hypreddef}.
\end{theorem}
}
The hyperbolic reduction $\bar{\Q}$ is isomorphic to a subscheme $Z$ of the relative Hilbert scheme of lines $M$ over $S$ parametrizing lines in the fibers of $p\colon \Q \to S$ that intersect the smooth section. Let $\P_Z(\R_Z) \subset \Q \times_S Z$ be the universal family of lines that $Z$ parametrizes. Using this identification, the embedding functor $\A_{\Q}\to \D(\Q)$ can be described explicitly as below.
{\renewcommand{\thetheorem}{\ref{main2}}
\begin{theorem}
In the hypotheses of Case 1, $\A_{\Q}\cong\D(Z)$ 
where $Z$ is introduced above. The embedding functors $ \D(Z)\to \D(\Q)$ are of Fourier-Mukai type with kernels $\cS_n^{\R_Z}, n\in\ZZ$ where $\cS_n^{\R_Z}$ is the $n$-th spinor sheaf with respect to the isotropic subbundle $\R_Z$ in Definition \ref{spinordef}. 
\end{theorem}
}
In Case 2, $\A_{\Q}$ is twisted geometric. 
{\renewcommand{\thetheorem}{\ref{main3}}
\begin{theorem}
Assume $\kk$ is algebraically closed and $\cchar(\kk)=0$. Let $p\colon \Q\to S$ be a flat quadric surface bundle where $\Q$ is smooth and $S$ is a smooth surface over $\kk$. Then $\A_{\Q}\cong \D(S^+, \A^+)$ where $S^+$ is the resolution of the double cover $\wt{S}$ over $S$ ramified along the (first) degeneration locus $S_1$ and $\A^+$ is an Azumaya algebra on $S^+$. In addition, the Brauer class $[\A^+]\in \Br(S^+)$ is trivial if and only if $p\colon \Q \to S$ has a rational section.
\end{theorem}
}

When the quadric surface bundle $p\colon \Q\to S$ has {\it simple degeneration}, i.e., fibers of $p$ have corank at most $1$, or equivalently $S_2=\emptyset$, it is well-known that $\A_{\Q}$ is equivalent to a twisted derived category of the double cover $\wt{S}$ over $S$ ramified along $S_1$. Furthermore, the twist on $\wt{S}$ is closely related to the relative Hilbert scheme of lines $M$. More precisely, $\rho\colon M\to S$ factors as a smooth conic bundle $\tau\colon M\to \wt{S}$ followed by the double cover $\alpha\colon \wt{S}\to S$. The twist on $\wt{S}$ is given by the Azumaya algebra corresponding to $\tau$. This correspondence is a relative version of that between central simple algebras and Severi-Brauer varieties (Theorem 2.4.3, 5.2.1 in \cite{csa}).

When $p\colon \Q\to S$ has fibers of corank $2$, $\A_{\Q}$ becomes more complicated and this paper focuses on this case. Having fibers of corank $2$ gives two challenges. On one hand, the singular locus of a quadric surface $Q$ of corank $2$ over an algebraically closed field is not isolated (isomorphic to $\P^1$). This means that $\A_Q$ would be ``$1$-dimensional" because it {\it absorbs the singularity} of $Q$ (Its semiorthogonal complement in $\D(Q)$ is an exceltional collection.). On the other hand, the Hilbert scheme of lines on $Q$ is the union of two $\P^2$'s intersecting at a point, which is reducible and has higher dimension than cases for quadric surfaces of corank at most $1$. 

\begin{conj}\label{conj}
    If a quadric surface bundle $p\colon \Q\to S$ has simple degeneration generically and each fiber has corank at most $2$, i.e., $S_2\neq S$ and $S_3=\emptyset$, then $\A_{\Q}\cong \D(Y, \alpha_Y)$ where $(Y, \alpha_Y)$ is some twisted scheme.
\end{conj}

We expect the conjecture to be true because \'{e}tale locally $p$ has a smooth section (Case 1). More specifically, we expect that $Y$ is a double cover over $S\backslash S_1$ and a $\P^1$-bundle over $S_2$. This is supported by Theorem \ref{main1}, \ref{main2} and Theorem \ref{main3} in the paper. So far we do not have a natural way to construct such $Y$ in general. But in the two cases considered in the paper, we can make such constructions. In the remaining part of the introduction, we will describe $Y$ more explicitly and give an overview of the techniques used in the paper.

In Case 1 where $p\colon \Q\to S$ is assumed to have a smooth section, $Y$ can be constructed as the hyperbolic reduction $\bar{\Q}$ with respect to the smooth section. The relative linear projection of $\Q$ from the smooth section identifies the blow-up of $\Q$ along the smooth section with a blow-up along $\bar{\Q}$; see the diagram (\ref{hrdiag}). The condition $S_2\neq S$ ensures that the blow-up center $\bar{Q}$ is a regular embedding of codimension $2$, and then we can apply the blow-up formula \cite[Theorem 6.11]{jquot-1}. Theorem \ref{main1} follows from performing mutations under this identification. This proof is straightforward, but it has the disadvantage that the information on the embedding functor $\A_{\Q}\to \D(\Q)$ is lost under mutations. This is why we provide a second proof in Theorem \ref{main2} and focus on working with $\A_{\Q}$ by itself. 

In the second proof of Case 1, we make use of the isomorphism from $\bar{\Q}$ to the scheme $Z$ over $S$ parametrizing lines in the fibers of $p\colon \Q \to S$ that intersect the smooth section. It has been shown in \cite{kuzqfib, abbqfib} that $\A_{\Q}$ is equivalent to $\D(S, \B_0)$, the derived category of coherent sheaves on $S$ with right $\B_0$-module structures, where $\B_0$ is the even Clifford algebra of $p$. We prove in Proposition \ref{nceqprop} that $\D(S,\B_0)\cong \D(Z, \sEnd(\I_0^{\R_Z}))$ where $\I_0^{\R_Z}$ is certain Clifford ideal introduced in Section \ref{clalid}. Finally, we use Morita equivalence to deduce $\A_{\Q}\cong \D(Z)$. The embedding functors $\D(Z)\to \D(\Q)$ can be described explicitly because each functor involved can be described so. This approach relies on the study of derived categories of non-commutative schemes and we provide the necessary foundations in Appendix \ref{ncschsec}.

In Case 2 where we assume $\Q$ is smooth and $S$ is a smooth surface, $Y$ can be constructed as the resolution $S^+$ of the double cover $\wt{S}$ over $S$. The idea is to make use of the relation between the relative Hilbert scheme of lines $M$ and the residual category $\A_{\Q}$ when $p\colon \Q \to S$ has simple degeneration. Although this relation fails for quadric surfaces of corank $2$, which means that $\A_{\Q}$ described in terms of the map $\tau\colon M\to \wt{S}$ is no longer a twisted derived category, we can fix this by making a modification to $\tau$. Namely, we will construct a smooth conic bundle $\tau_+\colon M^+\to S^+$ from $\tau$ such that $\A_{\Q}\cong \D(S^+, \A^+)$ where $\A^+$ is Brauer equivalent to the Azumaya algebra corresponding to $\tau_+$. This construction is motivated by the work \cite{kuzline}, in which the base $S$ is a smooth $3$-fold. We will not need the additional assumptions of \cite{kuzline} on the degeneration loci $S_k$. 

Now we give a more detailed description of the modification. We will assume $\kk$ is algebraically closed and $\cchar(\kk)=0$ so that we can perform birational transformations to $M$. In this case, there are only a finite number of fibers with corank $2$ by Lemma \ref{bealem}. As pointed out before, the fiber $M_s, s\in S_2$ is a union of two $\P^2$'s. For each $M_s$, we will blow up one of the $\P^2$'s and contract the exceptional locus onto $\P^1$. In this process, $M_s$ becomes a Hirzebruch surface $M^+_s$, and $M^+\to \wt{S}$ factor as a smooth conic bundle $\tau_+\colon M^+\to S^+$ followed by the resolution $S^+\to \wt{S}$. Details of this process are given in Proposition \ref{birtran}. It should be pointed out that this geometric construction relies on being able to choose one of the two $\P^2$'s for each $M_s, s\in S_2$. Hence, it works well when $S_2$ is a finite set of points, but it would not work in general.

The main results of the paper can be used to reprove results on semiorthogonal decompositions for the nodal quintic del Pezzo threefolds (Example \ref{nodaldp5}) given in \cite{xienodaldp5} and cubic $4$-folds containing a plane in non-generic cases (Example \ref{cubic4}). Since the residual category of a Fano complete intersection of quadrics is equivalent to that of the associated net of quadrics by the Homological Projective Duality theory, we can produce new examples of Fano complete intersections of quadrics whose residual categories are twisted geometric. For example, residual categories of smooth complete intersections of three quadrics in $\P^{2m+3}$ for $m\leqslant 5$ are twisted geometric; see Proposition \ref{ci3q}.

\medskip
\noindent{\bf Convention.}
Throughout the paper, we assume that $\kk$ is a field with $\cchar(\kk)\neq 2$ and $S$ is an integral noetherian scheme over $\kk$ unless specified otherwise. In Section \ref{surfbsec} (Case 2), we assume that $\kk$ is algebraically closed, and starting from Proposition \ref{birtran} till the end of the section, we assume additionally $\cchar{\kk}=0$.

\medskip
\noindent{\bf Acknowledgements.}
I would like to thank Arend Bayer, Qingyuan Jiang for numerous helpful conversations, and thank Alexander Kuznetsov for pointing out the reference \cite{moc8}. I would also like to thank the referee for a very careful reading of the paper and for many useful suggestions. The author is supported by the ERC Consolidator grant WallCrossAG, no. 819864.
\section{Quadric bundles, Clifford algebras and ideals}\label{prisec}
In this section, we recall some notions of quadric bundles, review definitions of Clifford algebras and introduce Clifford ideals. 

\subsection{Some basic notions of quadric bundles}\label{qbnotion}
Assume that $\E$ is a vector bundle and $\L$ is a line bundle on $S$. We say $q\colon \E\to \L$ is a \textit{(line-bundle valued) quadratic form} on $S$ if $q$ is an $\O_S$-homogeneous morphism of degree $2$ such that the associated morphism $b_q\colon \E\times \E\to \L$ defined by $b_q(v,w)=q(v+w)-q(v)-q(w)$ is a symmetric bilinear form. The \textit{rank} of $q\colon \E\to\L$ is the rank of $\E$.

Assume $q$ is non-zero. Let $\pi\colon \P_S(\E) \to S$ be the projection map. Then $q$ corresponds to a non-zero section
\begin{equation}\label{qsec}
    s_q\in \Gamma(\P_S(\E), \O_{\P_S(\E)/S}(2)\otimes \pi^*\L) \cong \Gamma(S, \Sym^2(\E^\vee)\otimes \L),
\end{equation}
where $\E^\vee$ is the dual of $\E$ and $\Sym^2$ is the second symmetric product. Let $\Q \subset \P_S(\E)$ be the zero locus of the section $s_q$. We write $\Q =\{q=0\}$ and call $p\colon \Q\to S$ the associated \textit{quadric bundle}. We say that $q\colon \E\to \L$ is \textit{primitive} if $q$ is non-zero over the residue field of every point on $S$. This is equivalent to $p\colon \Q\to S$ being a flat quadric bundle. In this paper, when we consider quadric bundles, we only require that $\Q \subset \P_S(\E)$ has codimension $1$ and thus $p\colon \Q\to S$ may not be flat.

Denote the \textit{$k$-th degeneration locus} of $p\colon \Q\to S$ by $S_k \subset S$, which is the closed subscheme of $S$ defined by the sheaf of ideal
\begin{equation*}
    \IIm(\Lambda^{n+1-k}\E \otimes \Lambda^{n+1-k} \E \otimes (\L^\vee)^{n+1-k} \xrightarrow{\Lambda^{n+1-k} b_q} \O_S),
\end{equation*}
where $n=\rank(\E)$. This means that $S_k$ is the locus where fibers of $p$ have corank at least $k$. The \textit{corank} of a quadric is the corank of its associated symmetric bilinear form. In particular, $S_1 \cong \{\det(b_q)=0\}$ is the locus of singular fibers. We say that $p\colon \Q \to S$ has \textit{simple degeneration} if $S_2=\emptyset$.

A subbundle $\cW$ of $q\colon \E\to\L$ is isotropic if $q|_{\cW}=0$. This is equivalent to $\P(\cW)\subset \Q$.
\begin{defn}\label{regisodef}
An isotropic subbundle $\cW$ of $q\colon \E\to\L$ is called \textit{regular} if for each geometric point $x\in S$, the fiber $\P(\cW)_x$ over $x$ is contained in the smooth locus of the fiber $\Q_x$. We call $\P_S(\cW)$ a \textit{smooth $r$-section} of $p\colon\Q\to S$ if $\cW$ is a regular isotropic subbundle of $q$ of rank $r+1$. A smooth $0$-section is simply called a \textit{smooth section}. 
\end{defn}

Let $\cW$ be a regular isotropic subbundle of $q\colon \E\to\L$. Then there is an exact sequence
\begin{equation*}
    0\to \cW^\perp \to \E \xrightarrow{b_q|_{\E\times \cW}} \sHom(\cW, \L) \to 0,
\end{equation*}
where $\cW^\perp$ is the kernel of $b_q|_{\E\times \cW}$. Since $\cW$ is isotropic, we have $\cW \subset \cW^\perp$ and $\cW$ is contained in the kernel of $q|_{\cW^\perp}\colon \cW^\perp \to \L$. It induces a new quadratic form $\bar{q}\colon \cW^\perp/\cW \to \L$.
\begin{defn}\label{hypreddef}
Denote $\bar{\E}=\cW^\perp/\cW$. The induced quadratic from $\bar{q}\colon \bar{\E} \to \L$ is called the \textit{hyperbolic reduction} of $q\colon \E \to \L$ with respect to the regular isotropic subbundle $\cW$. Alternatively, $\bar{\Q}=\{\bar{q}=0\}$ is called the \textit{hyperbolic reduction} of $\Q=\{q=0\}$ with respect to the smooth $r$-section $\P_S(\cW)$, where $r=\rank(\cW)-1$.
\end{defn}
The quadric bundles $\Q\to S$ and $\bar{\Q}\to S$ related by the hyperbolic reduction share many features. For example, they have the same degeneration loci $S_k$.
\subsection{Clifford algebras and ideals}\label{clalid}
Let $q\colon\E\to\L$ be a non-zero quadratic form and let $p\colon \Q=\{q=0\} \to S$ be the associated quadric bundle. There are several equivalent definitions for even Clifford algebras and Clifford bimodules (also called the odd part of Clifford algebras) of $q$; see \cite[\S3]{bkqf}, \cite[\S1.5]{abbqfib}, \cite[\S3.3]{kuzqfib}. We will recall the construction in \cite{bkqf}.

Set the degrees of elements of $\E$ and $\L$ to be $1$ and $2$, respectively. The \textit{generalized Clifford algebra} is defined by
\begin{equation}
    \B= T(\E)\otimes (\bigoplus_{n\in \ZZ} \L^n)/\langle (v\otimes v)\otimes 1- 1\otimes q(v)\rangle_{v\in \E},
\end{equation}
where $T(\E)$ is the tensor algebra of $\E$. Let $\B_n$ be the subgroups of $\B$ consisting of elements of degree $n\in\ZZ$. Then $\B \cong \bigoplus_{n\in\ZZ} \B_n$  is a $\ZZ$-graded algebra over $\B_0$. The \textit{even Clifford algebra} and the \textit{Clifford bimodule} are defined to be $\B_0$ and $\B_1$, respectively. 

Write $\rank(\E)=2m$ or $2m+1$. Then there are $\O_S$-filtrations
\begin{equation}\label{b01fil}
    \begin{split}
    & \O_S=F_0\subset F_2\subset \dots \subset F_{2m}=\B_0,\\
    & \E= F_1\subset F_3\subset \dots \subset F_{2m+1} =\B_1\\
\end{split}
\end{equation}
such that $F_{2i}/F_{2i-2} \cong \Lambda^{2i}\E \otimes (\L^\vee)^i$ and $F_{2i+1}/F_{2i-1}\cong \Lambda^{2i+1} \E \otimes (\L^\vee)^i$. Moreover, we have
\begin{equation}\label{bn}
    \B_n\cong \left\{
    \begin{array}{ll}
    \B_0\otimes \L^k, & n=2k\\
    \B_1\otimes \L^k, & n=2k+1\\   
    \end{array}
\right..
\end{equation}
When $q$ is primitive, the Clifford multiplications
\begin{equation}\label{primiso}
    \B_n\otimes_{\B_0}\B_m\xrightarrow{\cong} \B_{n+m}
\end{equation}
are isomorphisms for $n,m\in\ZZ$ and $\B_1$ is an invertible $\B_0$-bimodule. 

Note that if $\cW$ is an isotropic subbundle, then $\bigoplus_n \Lambda^n \cW$ is a subalgebra of $\B$. We define the associated Clifford ideals below.

\begin{defn}\label{clideal}
The \textit{$n$-th (left) Clifford ideal} $\I_n^{\cW}$ of $q\colon\E\to \L$ with respect to an isotropic subbundle $\cW$ is the degree $n$ part of the left principal ideal of $\B$ generated by $\det\cW\subset \B_{\rank(\cW)}$, i.e.,
\begin{equation*}
   \I_n^{\cW} =\IIm(\B_{n-\rank(\cW)}\otimes \det\cW \to \B_n),
\end{equation*}
where the map is given by Clifford multiplications. Similarly, we can define the \textit{$n$-th right Clifford ideal}
\begin{equation*}
    \I_n^{\circ \cW}= \IIm(\det\cW\otimes \B_{n-\rank(\cW)} \to \B_n).
\end{equation*}
\end{defn}
We have $\I_n^{\cW}\cong\I_n^{\circ \cW}$ as vector bundles on $S$ because tensor is commutative and $\I_{n+2}^\cW \cong \I_n^\cW\otimes \L$ by relations (\ref{bn}).

The Clifford ideals play an important role in the study of quadric bundles. They have been studied for a quadric hypersurface and were used to define spinor sheaves for an isotropic subspace in \cite[\S2]{addspinor}. In this paper, we will provide a relative version. 

Firstly, we provide two lemmas about Clifford ideals that will be used later. Next, we give the relation between Clifford ideals of different isotropic subbundles and describe the dual of Clifford ideals. As with the quadric hypersurface case, we can define spinor sheaves for an isotropic subbundle, and we will discuss properties of spinor sheaves that follow from those of Clifford ideals. We primarily work with left Clifford ideals and properties proved for them also apply to right Clifford ideals. 

\begin{lemma}\label{mulisolem}
Assume that $q\colon \E\to \L$ is primitive and $\cW$ is an isotropic subbundle. Then for all $m,n\in\ZZ$,

(1) the Clifford multiplication induces a left $\B_0$-module isomorphism
\begin{equation*}
    \sigma_{m,n} \colon \B_m\otimes_{\B_0} \I_n^{\cW} \xrightarrow{\cong} \I_{m+n}^{\cW};
\end{equation*}

(2) there are isomorphisms of sheaves of algebras  $\sEnd(\I_n^{\cW})\cong \sEnd(\I_{m+n}^{\cW})$.

A similar result is true for right Clifford ideals.
\end{lemma}
\begin{proof}
(1) We have the map $\sigma_{m,n}$ because the image of the Clifford multiplication $\B_m\otimes \I_n^{\cW} \to \B_{m+n}$ is $\I_{m+n}^{\cW}$ and it factors through $\B_m\otimes_{\B_0} \I_n^{\cW}$. Applying $\B_m\otimes_{\B_0} -$ to the map 
\begin{equation*}
    \sigma_{-m,m+n}\colon \B_{-m}\otimes_{\B_0} \I_{m+n}^{\cW} \to \I_n^{\cW}    
\end{equation*}
and using the isomorphism (\ref{primiso}), we obtain the map $\I_{m+n}^{\cW}\to \B_m\otimes_{\B_0} \I_n^{\cW}$. This is the inverse map of $\sigma_{m,n}$.

(2) is a consequence of (1) where the morphism $\sEnd(\I_n^{\cW})\to \sEnd(\I_{m+n}^{\cW})$ is induced by $\B_m\otimes_{\B_0} -$.
\end{proof}

\begin{lemma}\label{coklem}
Let $\cW$ be an isotropic subbundle of rank $r$. Then for all $n\in \ZZ$ there are exact sequences
\begin{equation*}
    \B_{n-1}\otimes \cW \to \B_n \to \I_{n+r}^{\cW}\otimes \det(\cW^\vee) \to 0
\end{equation*}
of left $\B_0$-modules, where the first map is given by the Clifford multiplication $\B_{n-1}\otimes \cW \subset \B_{n-1}\otimes \B_1 \to \B_n$. A similar result is true for right Clifford ideals.
\end{lemma}
\begin{proof}
There is a left $\B_0$-module surjection $\B_n\twoheadrightarrow \I_{n+r}^{\cW}\otimes \det(\cW^\vee)$ induced by the multiplication $\B_n\otimes \det(\cW)\subset \B_n\otimes \B_r\to \B_{n+r}$. By construction, the kernel of the surjection is given by the image of $\B_{n-1}\otimes \cW \to \B_n$.
\end{proof}

\begin{lemma}\label{2clideal}
Let $\cW'\subset \cW$ be isotropic subbundles and assume $\rank(\cW)-\rank(\cW')=1$. Let $\L_1=\cW/\cW'$. Then for all $n\in \ZZ$ there are short exact sequences
\begin{equation*}
    0\to \I_n^{\cW}\to \I_n^{\cW'}\to \I_{n+1}^{\cW}\otimes \L_1^\vee\to 0
\end{equation*}
of left $\B_0$-modules. Similarly, if $\cW$ is an isotropic sub line bundle, then for all $n\in \ZZ$ there are short exact sequences
\begin{equation*}
    0\to \I_n^{\cW}\to \B_n\to \I_{n+1}^{\cW}\otimes \cW^\vee\to 0.
\end{equation*}
A similar result is true for right Clifford ideals.
\end{lemma}
\begin{proof}
Consider the map given by the Clifford multiplication
\begin{equation*}
    \I_n^{\cW'}\otimes \cW \subset \B_n \otimes \B_1 \to \B_{n+1}.
\end{equation*}
The image is $\I_{n+1}^\cW$ and it factors through $\I_n^{\cW'}\otimes (\cW/\cW')$. Furthermore, the induced map
\begin{equation*}
    \I_n^{\cW'}\otimes \L_1 \twoheadrightarrow \I_{n+1}^\cW
\end{equation*}
 has kernel $\I_n^\cW \otimes \L_1$. The proof when $\rank(\cW)=1$ is similar.
\end{proof}

\begin{lemma}\label{clidealdual}
Let $\cW$ be an isotropic subbundle of rank $r$.

(1) If $\rank(\E)=2m$, then for $k\in\ZZ$ there are right $\B_0$-module isomorphisms
\begin{equation*}
    (\I_k^\cW)^\vee \cong \I_{r-k}^{\circ \cW}\otimes \det(\cW^\vee)\otimes \det(\E^\vee)\otimes \L^m.
\end{equation*}

(2) If $\rank(\E)=2m+1$, then for $k\in\ZZ$ there are right $\B_0$-module isomorphisms
\begin{equation*}
    (\I_k^\cW)^\vee \cong \I_{r+1-k}^{\circ \cW}\otimes \det(\cW^\vee)\otimes \det(\E^\vee)\otimes \L^m.
\end{equation*}
\end{lemma}
\begin{proof}
(1) Let $\tr\colon \B_0\to \det(\E)\otimes (\L^\vee)^m$ be the map $F_{2m}\to F_{2m}/F_{2m-2}$ induced by the filtration (\ref{b01fil}). For $k\in\ZZ$, there is a pairing
\begin{equation*}
    \begin{array}{rcccl}
    \B_{-k}\otimes \B_{k} &\to & \B_0 & \xrightarrow{\tr} & \det(\E)\otimes (\L^\vee)^m,\\
    \xi \otimes \eta & \mapsto & \xi\eta & \mapsto & \tr(\xi\eta).
    \end{array}
\end{equation*}
When the pairing is restricted to $\B_{-k}\otimes \I_k^\cW$, it induces
\begin{equation*}
    f\colon \B_{-k} \to \sHom(\I_k^\cW, \det(\E)\otimes (\L^\vee)^m).
\end{equation*}
It is clear from the construction that $f$ is a homomorphism of right $\B_0$-modules. On the other hand, by Lemma \ref{coklem}, we get a right $\B_0$-module surjection 
\begin{equation*}
    g\colon \B_{-k}\twoheadrightarrow \det(\cW^\vee)\otimes \I_{r-k}^{\circ \cW}.
\end{equation*}
Since $f, g$ have the same kernel,  we have an injection
\begin{equation*}
    \bar{f}\colon \det(\cW^\vee)\otimes \I_{r-k}^{\circ \cW} \to \sHom(\I_k^\cW, \det(\E)\otimes (\L^\vee)^m).
\end{equation*}
Note both the vector bundles above have rank $2^{\rank(\E)-r-1}$. Then $\bar{f}$ is an isomorphism because it is so over the residue field of every point on $S$.

(2) The proof when $\rank(\E)$ is odd is similar. The only differences are that we should use $\tr\colon \B_1\to \det(\E)\otimes (\L^\vee)^m$ and the pairing $\B_{1-k}\otimes \B_{k}\to \det(\E)\otimes (\L^\vee)^m$ instead.
\end{proof}

Now we define spinor sheaves on quadric bundles by means of Clifford ideals. Let $\pi\colon \P_S(\E)\to S$ be the projection map. Regard $\O_{\P_S(\E)/S}(-1)$ as the universal sub line bundle and consider the map
\begin{equation}\label{phi_n}
    \begin{array}{rcl}
    \O_{\P_S(\E)/S}(-1)\otimes \pi^*\I_{n-1}^\cW & \xrightarrow{\phi_n} & \pi^* \I_n^\cW,\\
    v\otimes \xi & \mapsto & v \xi.
    \end{array} 
\end{equation}
Then $\phi_n\circ \phi_{n-1}=q$. Taking into account that $S$ is integral and $q\neq 0$, we have $\phi_n$ is an isomorphism outside of $\Q=\{q=0\}$ and $\phi_n$ is injective.

\begin{defn}\label{spinordef}
The $n$-th \textit{spinor sheaf} $\cS_n^\cW$ on $\Q=\{q=0\}$ of the non-zero quadratic form $q\colon\E\to\L$ with respect to an isotropic subbundle $\cW$ is defined by the exact sequence
\begin{equation*}
    0\to \O_{\P_S(\E)/S}(-1)\otimes \pi^*\I_{n-1}^\cW \xrightarrow{\phi_n} \pi^*\I_n^\cW \to i_* \cS_n^\cW \to 0,
\end{equation*}
where $\phi_n$ is constructed in (\ref{phi_n}), $\pi\colon \P_S(\E)\to S$ is the projection map and $i\colon \Q\hookrightarrow \P_S(\E)$ is the embedding.
\end{defn}
Again we have $\cS_{n+2}^\cW \cong \cS_n^\cW \otimes \L$.

\begin{remark}\label{spinordef2}
We can also construct $\cS_n^\cW$ as the cokernel of $\phi_n^\circ$ where 
\begin{equation*}
    \phi_n^\circ\colon \O_{\P_S(\E)/S}(-1)\otimes \pi^*\I_{n-1}^{\circ\cW} \to \pi^*\I_n^{\circ\cW}
\end{equation*}
is the map sending $v\otimes \xi$ to $\xi v$.
\end{remark}

Restricting $\phi_n, n\in\ZZ$ to the quadric bundle $\Q$, there are exact sequences
\begin{equation}\label{spinorseq}
    \begin{split}
    & \dots \to \O_{\Q/S}(-2)\otimes p^*\I_{n-2}^\cW \xrightarrow{\phi_{n-1}} \O_{\Q/S}(-1)\otimes p^*\I_{n-1}^\cW \xrightarrow{\phi_n} p^*\I_n^\cW \to \cS_n^\cW \to 0,\\
    & 0\to \cS_n^\cW \to \O_{\Q/S}(1)\otimes p^*\I_{n+1}^\cW \xrightarrow{\phi_{n+2}} \O_{\Q/S}(2)\otimes p^*\I_{n+2}^\cW \xrightarrow{\phi_{n+3}} \dots,
    \end{split}
\end{equation}
where $p\colon \Q\to S$ is the quadric bundle.
\begin{cor}
Let $\cW'\subset \cW$ be isotropic subbundles and assume $\rank(\cW)-\rank(\cW')=1$. Let $\L_1=\cW/\cW'$. Then for all $n\in \ZZ$ there are short exact sequences on $\Q$,
\begin{equation*}
    0\to \cS_n^{\cW}\to \cS_n^{\cW'}\to \cS_{n+1}^{\cW}\otimes p^*\L_1^\vee\to 0,
\end{equation*}
where $p\colon \Q\to S$ is the quadric bundle.
\end{cor}
\begin{proof}
By Lemma \ref{2clideal}, there are short exact sequences
\begin{equation*}
    0\to \I_n^{\cW}\to \I_n^{\cW'}\to \I_{n+1}^{\cW}\otimes \L_1^\vee\to 0.
\end{equation*}
Then the result follows because these short exact sequences are compatible with the map $\phi_n$ (\ref{phi_n}) defining the spinor sheaves.
\end{proof}

\begin{cor}
The spinor sheaf $\cS_n^\cW$ is reflexive on $\Q$. Let $\rank(\cW)=r$. Recall $p\colon \Q\to S$ is the quadric bundle.

(1) If $\rank(\E)=2m$, then 
\begin{equation*}
    (\cS_n^\cW)^\vee \cong \cS_{r-n-1}^{\cW} \otimes \O_{\Q/S}(-1) \otimes p^*(\det(\cW^\vee)\otimes \det(\E^\vee)\otimes \L^m).
\end{equation*}

(2) If $\rank(E)=2m+1$, then
\begin{equation*}
    (\cS_n^\cW)^\vee \cong \cS_{r-n}^{\cW} \otimes \O_{\Q/S}(-1) \otimes p^*(\det(\cW^\vee)\otimes \det(\E^\vee)\otimes \L^m).
\end{equation*}
\end{cor}
\begin{proof}
The reflexivity of $\cS_n^\cW$ follows from the observation that taking double dual of (\ref{spinorseq}) gives the same exact sequences.

(1) Taking the dual of the second sequence in (\ref{spinorseq}), we have another exact sequence
\begin{equation*}
    \dots \to \O_{\Q/S}(-2)\otimes p^*(\I_{n+2}^\cW)^\vee \xrightarrow{\phi_{n+2}^\vee} \O_{\Q/S}(-1) \otimes p^*(\I_{n+1}^\cW)^\vee \to (\cS_n^\cW)^\vee \to 0.
\end{equation*}
The result follows from Lemma \ref{clidealdual} and Remark \ref{spinordef2} by noticing that $\phi_{n+2}^\vee \cong \phi_{r-n-1}^\circ$. (2) can be proved similarly.
\end{proof}

\subsection{Non-primitive quadratic forms of rank two}\label{qfr2}
Given a quadric surface bundle with a smooth section, the hyperbolic reduction with respect to the smooth section as in Definition \ref{hypreddef} gives a quadratic form of rank $2$. If $S_2\neq \emptyset$, then the new quadratic form is non-primitive. As a preparation for Section \ref{smoothsec}, we will study possibly non-primitive quadratic forms of rank $2$ in this section.

Let $q\colon\E\to\L$ be a non-zero quadratic form of rank $2$, i.e., $\rank(\E)=2$, and let $p\colon \Q=\{q=0\} \to S$ be the associated quadric bundle. Then $q|_{S_2}=0$, and $q\neq 0$ implies $S_2\neq S$. In this case, $p$ has relative dimension $0$ over $S\backslash S_2$ and is not flat if $S_2\neq \emptyset$. We will give a result on the relation between Clifford algebras and Clifford ideals. The non-flatness of $p$ requires a non-trivial argument.

We observe that $\O_{\Q/S}(-1)$ is an isotropic line bundle of $p^*q\colon p^*\E\to p^*\L$. The Clifford multiplication gives 
\begin{equation*}
    p^*\B_n\otimes \I_m^{\O_{\Q/S}(-1)}\to \I_{n+m}^{\O_{\Q/S}(-1)}.
\end{equation*}
Since $\B_n$ is locally free, the map above induces
\begin{equation}\label{hommap}
    Lp^*\B_n \cong p^*\B_n \to \sHom_{\O_{\Q}}(\I_m^{\O_{\Q/S}(-1)}, \I_{n+m}^{\O_{\Q/S}(-1)}).
\end{equation}

\begin{lemma}\label{dim0lem}
Let $p\colon \Q\to S$ be the quadric bundle associated with a non-zero quadratic form  $q\colon \E\to \L$ of rank $2$. Let $\pi\colon \P_S(\E)\to S$ be the projection map and let $i\colon \Q\hookrightarrow \P_S(\E)$ be the embedding. Then $p=\pi\circ i$.

(1) There is a short exact sequence
\begin{equation*}
    0\to \pi^*\B_{n-1} \otimes \O_{\P_S(\E)/S}(-1)\xrightarrow{\phi_n} \pi^*\B_n \to i_*\sHom_{\O_{\Q}}(\I_1^{\O_{\Q/S}(-1)}, \I_{n+1}^{\O_{\Q/S}(-1)})\to 0
\end{equation*}
where $\phi_n$ is the map (\ref{phi_n}).

(2) The map 
\begin{equation*}
    \B_n\to Rp_*\sHom_{\O_{\Q}}(\I_m^{\O_{\Q/S}(-1)}, \I_{n+m}^{\O_{\Q/S}(-1)})
\end{equation*}
induced by (\ref{hommap}) is an isomorphism for all $n,m\in \ZZ$.
\end{lemma}
\begin{proof}
(1) Since $S$ is integral, $\phi_n\circ \phi_{n-1}=q$ and $q\neq 0$, we have $\phi_n$ is injective and an isomorphism outside of $\Q$. Moreover, $\coker(\phi_n)$ is supported on $\Q$ schematically and we can write it as $i_*\fS_n$. Note that $I_1^{\O_{\Q/S}(-1)}\cong \O_{\Q/S}(-1)$ and $\fS_n\cong \coker(\phi_n|_{\Q})$. By Lemma \ref{coklem}, we get
\begin{equation*}
    \fS_n\cong \sHom_{\O_{\Q}}(\I_1^{\O_{\Q/S}(-1)}, \I_{n+1}^{\O_{\Q/S}(-1)}).
\end{equation*}

(2) It suffices to prove the claim for $m=0,1$ and all $n$. Applying $R\pi_*$ to the sequence in (1), we get the claim holds for $m=1$.

Now we prove for $m=0$. Let $F=\coker(\O_{\P_S(\E)/S}(-1)\to \pi^*\E)\cong \pi^*\det(\E) \otimes \O_{\P_S(\E)/S}(1)$. From (1), we get
\begin{equation*}
    \fS_1\cong \sHom_{\O_{\Q}}(\I_1^{\O_{\Q/S}(-1)}, \I_2^{\O_{\Q/S}(-1)}) \cong p^*\det\E\otimes \O_{\Q/S}(1).
\end{equation*}
Comparing $F$ and $\fS_1$, we get $\fS_1 \cong F|_{\Q}$. We have commutative diagrams with exact rows
\begin{equation*}
    \begin{tikzcd}
    0\arrow{r} & \pi^*\B_n \otimes \O_{\P_S(\E)/S}(-1)\arrow{r} \arrow[equal]{d} & \pi^*\B_n \otimes \pi^*\E \arrow{r} \arrow{d}{f} & \pi^*\B_n\otimes F \arrow{r} \arrow{d}{g} & 0\\
    0 \arrow{r} & \pi^*\B_n \otimes \O_{\P_S(\E)/S}(-1) \arrow{r} & \pi^*\B_{n+1} \arrow{r} & i_*\fS_{n+1} \arrow{r} & 0,
    \end{tikzcd}
\end{equation*}
where $f$ is pull-back of the Clifford multiplication $\B_n\otimes \E \to \B_{n+1}$ and $g$ is the map induced by the diagram. Tensoring $g$ by $F^\vee$ gives
\begin{equation*}
    \pi^*\B_n \to i_*\fS_{n+1}\otimes F^\vee \cong i_* \sHom_{\O_{\Q}}(\fS_1, \fS_{n+1}) \cong i_* \sHom_{\O_{\Q}}(\I_0^{\O_{\Q/S}(-1)}, \I_n^{\O_{\Q/S}(-1)}).
\end{equation*}
Clearly, $f$ is surjective. Thus, $g$ is also surjective and $\ker(f)\cong \ker(g)$. Observe that $R\pi_*(\ker(f)\otimes \F^\vee)=0$. Hence, $R\pi_*(\ker(g)\otimes \F^\vee)=0$ and the claim holds for $m=0$.
\end{proof}
\section{Relative Hilbert schemes of lines}\label{hilbsec}
In this section, we focus on describing the relative Hilbert schemes of lines of flat quadric surface bundles. 

In the rest of the paper, we will use the following notations.
\begin{itemize}
    \item Let $p\colon \Q\to S$ be a flat quadric surface bundle with the associated quadratic form $q\colon \E\to \L$ and the generalized Clifford algebra $\B \cong \bigoplus_{n\in\ZZ} \B_n$.
    \item Let $\pi\colon \P_S(\E) \to S$ be the projection map and let $i\colon \Q\hookrightarrow \P_S(\E)$ be the embedding. Then $p=\pi\circ i$.
    \item  Let $\wt{S} =\Spec_S(\cC_0) $ be the \textit{discriminant cover} over $S$ where $\cC_0$ is the center of $\B_0$. Denote the covering map by $\alpha\colon \wt{S} \to S$. The description of $\alpha$ is given in Lemma \ref{disccover} below.
    \item Let $\rho \colon M\to S$ be the relative Hilbert scheme of lines of $p$ and it factors as $M \xrightarrow{\tau}\wt{S}\xrightarrow{\alpha} S$.
\end{itemize}

Observe that $M\subset \Gr_S(2, \E)$. Let $\R$ be the universal subbundle on $\Gr_S(2, \E)$. Then $M \subset \Gr_S(2,\E)$ is the zero locus of the section
\begin{equation}\label{msec}
    s_{M} \in \Gamma(\Gr_S(2,\E), \Sym^2(\R^\vee)\otimes \pi_{\Gr}^*\L) \cong \Gamma(S, \Sym^2(\E^\vee)\otimes \L)
\end{equation}
corresponding to the section $s_q$ in (\ref{qsec}) defining $\Q$. Here $\pi_{\Gr}\colon \Gr_S(2, \E)\to S$ is the projection map. 

For a geometric point $s\in S$, the fiber $M_s$ is 
\begin{enumerate}
    \item a disjoint union of two smooth conics if $\Q_s$ is smooth; 
    \item a smooth conic over the dual numbers $\kk[\epsilon]/\epsilon^2$ if $\Q_s$ has corank $1$;
    \item a union of two planes intersecting at a point if $\Q_s$ has corank $2$.
\end{enumerate}

When $p$ has a smooth section as in Definition \ref{regisodef}, we further denote by
\begin{itemize}
    \item $Z\subset M$ the subscheme parametrizing lines that intersect the smooth section and $\beta: Z\hookrightarrow M\xrightarrow{\rho} S$ the composition map.
    \item $\R_Z$ the restriction of the universal subbundle $\R$ on $\Gr_S(2,\E)$ to $Z$.
\end{itemize}

\medskip

We would like to point out that whether $\rho \colon M\to S$ factors through a double cover is a subtle question in positive characteristic. This is the case when $\cchar(\kk)\neq 2, 3$ because there is the decomposition
\begin{equation*}
    \B_0\cong \O_S \oplus \Lambda^2 \E \otimes \L^\vee \oplus \det(\E)\otimes (\L^\vee)^2.
\end{equation*}
In this case, $\widehat{\cC_0}:= \O_S\oplus \det(\E)\otimes (\L^\vee)^2$ is a subalgebra inside the center $\cC_0$ and $\widehat{\alpha}\colon \widehat{S}:= \Spec_S(\widehat{\cC_0})\to S$ is a double cover over $S$ ramified along $S_1$. The inclusion $\widehat{\cC_0}\subset \cC_0$ implies that $\rho\colon M\to S$ factors through $\widehat{\alpha}$. In general, it is unclear if we can embed $\widehat{\cC_0}$ inside $\cC_0$. But we still have the following when $\cchar(\kk)\neq 2$.

\begin{lemma}\label{disccover}
The center $\cC_0$ of $\B_0$ is locally free of rank $2$ in the following two cases:

(a) $S_2=\emptyset$ or

(b) $S$ is a locally factorial integral scheme and $S_1\neq S$.

In these cases, $\alpha\colon \wt{S}\to S$ is a double cover over $S$ ramified along the degeneration locus $S_1$.
\end{lemma}
\begin{proof}
(a) Let $s\in S$ be an arbitrary point. When $S_2=\emptyset$, Lemma 1.4 in \cite{apsqsb} implies that the stalk of $\E$ at $s$ has an orthogonal basis $\{v_i\}_{i=1}^4$, i.e., $b_q(v_i, v_j)= 0$ for $i\neq j$. Then $\{1, v_1v_2v_3v_4\}$ is contained in the stalk of $\cC_0$ at $s$, which gives $\dim_{\kk(s)}(\cC_0\otimes \kk(s))\geqslant 2$. On the other hand, $\B_0\otimes \kk(s)$ is an Azumaya algebra whose center has dimension $2$. Thus,  $\dim_{\kk(s)}(\cC_0\otimes \kk(s))=2$ and $\cC_0$ is locally free of rank $2$. 

The case (b) is Lemma 1.6.1 in \cite{abbqfib}.
\end{proof}

If $p$ has simple degeneration, i.e., $S_2=\emptyset$, then $\tau: M\to \wt{S}$ is a smooth conic bundle. When $p$ has a smooth section, the composition $Z\hookrightarrow M \xrightarrow{\tau} \wt{S}$ is an isomorphism and thus $\tau$ is a $\P^1$-bundle. In this case, there are well-known relations among the relative Hilbert scheme of lines $M$, the even Clifford algebra $\B_0$, and Clifford ideals $\I_n^{\R_Z}$.

\begin{lemma}\label{sdb0lem}
Assume that the flat quadric surface bundle $p\colon \Q\to S$ has a smooth section and $S_2=\emptyset$. Then we can identify $\beta \cong \alpha \colon Z\xrightarrow{\cong}\wt{S}\to S$. For all $n\in \ZZ$, we have

(1) $M \cong \P_Z(\I_n^{\R_Z})$;

(2) $\B_0 \cong \beta_* \sEnd(\I_n^{\R_Z}).$\\
Here $\I_n^{\R_Z}$ is the $n$-th Clifford ideal of $\beta^*q$ on $Z$ in Definition \ref{clideal}.
\end{lemma}
\begin{proof}
(1) Consider base changes along $\beta\colon Z\to S$,
\begin{equation*}
    \begin{tikzcd}
    \Q\times_S Z \arrow{r}{\beta_{\Q}} \arrow{d}{p_Z} & \Q \arrow{d}{p}\\
    Z \arrow{r}{\beta} & S,
    \end{tikzcd}
    \quad
    \begin{tikzcd}
    \P_S(\E)\times_S Z \arrow{r}{\beta_{\E}} \arrow{d}{\pi_Z} & \P_S(\E) \arrow{d}{\pi}\\
    Z \arrow{r}{\beta} & S,
    \end{tikzcd}
\end{equation*}
and let $i_Z\colon \Q\times_S Z\hookrightarrow \P_S(\E)\times_S Z$ be the embedding. By Definition \ref{spinordef}, we have
\begin{equation}\label{spinoronz}
    0 \to \beta_{\E}^*\O_{\P_S(\E)/S}(-1)\otimes \pi_Z^*\I_{n-1}^{\R_Z} \to  \pi_Z^*\I_{n}^{\R_Z} \to i_{Z*} \cS_n^{\R_Z} \to 0.
\end{equation}
Let  $z=[L]\in Z$ be a geometric point represented by a line $L$ in the fiber $\Q_{\beta(z)}$. Then the spinor sheaf $\cS_{0}^{\R_Z}$ restricted at $z$ is the line bundle (resp., rank $1$ reflexive sheaf) $\O_{\Q_{\beta(z)}}(L)$ when $\Q_{\beta(z)}$ is smooth (resp., of corank $1$). Therefore, $M$ is the projectivization of $p_{Z*}\cS_{0}^{\R_Z}$. Applying $\pi_{Z*}$ to the sequence (\ref{spinoronz}), we get $\I_n^{\R_Z} \cong p_{Z*}\cS_{n}^{\R_Z}$. Hence, $M \cong \P_Z(\I_0^{\R_Z})$.

From Lemma \ref{mulisolem} (2), we get that $\sEnd(\I_n^{\R_Z})$ are isormorphic as sheaves of algebras for all $n$. Thus, $\I_n^{\R_Z}$ for different $n$ differ only by tensoring with a line bundle. This proves (1).

(2) Lemma 4.2 in \cite{kuzcubic4} shows that the push-forward $\beta_*$ of the Azumaya algebra corresponding to the smooth conic bundle $\tau\colon M\to \wt{S}\cong Z$ is isomorphic to $\B_0$. More specifically, the left $\beta^*\B_0$-module structure of $\I_n^{\R_Z}$ gives $\beta^*\B_0\to \sEnd(\I_n^{\R_Z})$  and the map adjoint to it induces $\B_0 \cong \beta_*\sEnd(\I_n^{\R_Z})$. 
\end{proof}

For the rest of the section, we will show that Lemma \ref{sdb0lem} (2) still holds when $p\colon \Q\to S$ does not have simple degeneration. This is given by Corollary \ref{genb0cor}.

Assume that $p\colon \Q\to S$ has a smooth section and let $\N$ be the corresponding regular isotropic sub line bundle. Denote $\N^\perp/\N$ by $\bar{\E}$. Let $\bar{q}\colon \bar{\E} \to \L$ be the hyperbolic reduction of $q$ with respect to $\N$. By construction, each point on $\P_S(\bar{\E})$ corresponds to a line in the fiber of $\P_S(\N^\perp)\to S$ that intersects $\P_S(\N)$. Hence,
\begin{equation}\label{zeqhr}
    Z\cong \{\bar{q}=0\} \subset \P_S(\bar{\E})
\end{equation}
is the hyperbolic reduction of $\Q$. The short exact sequence
\begin{equation*}
    0\to \N\to \N^\perp \to \bar{\E}\to 0
\end{equation*}
induces inclusions $\N\otimes \bar{\E}\subset \Lambda^2 \N^\perp \subset \Lambda^2 \E$. Under the isomorphism $\P_S(\bar{\E}) \cong \P_S(\N\otimes \bar{\E}) \subset \P_S(\Lambda^2 \E)$, we have
\begin{equation}\label{detr}
    \det(\R_Z) \cong \beta^*\N \otimes \O_{Z/S}(-1),
\end{equation}
where $\O_{Z/S}(-1)$ is the restriction of $\O_{\P_S(\bar{\E})/S}(-1)$.

\begin{lemma}\label{genb0lem}
The following properties about $\beta\colon Z\to S$ hold.

(1) Let $Z'=\beta^{-1}(S_2)$ and let $\bar{\pi}'=\beta|_{Z'}\colon Z'\to S_2$. Denote $\bar{\E}|_{S_2}$ by $\bar{\E_2}$ and $\I_0^{\R_Z}|_{Z'}$ by $\I_0^{\R_{Z'}}$. Then $Z'\cong \P_{S_2}(\bar{\E_2})$ and
\begin{equation}\label{i0dec}
   \I_0^{\R_{Z'}} \cong \O_{Z'/S}(-1)\otimes \bar{\pi}'^* \left( (\N \otimes \L^\vee)|_{S_2} \right) \oplus \bar{\pi}'^* \left( (\det(\E)\otimes (\L^\vee)^2)|_{S_2} \right).
\end{equation}

(2) Assume $S_2\neq S$. There is an exact sequence
\begin{equation*}
    0\to \O_{\P_S(\bar{\E})/S}(-2)\otimes \bar{\pi}^*\L^\vee \to \O_{\P_S(\bar{\E})} \to \O_Z\to 0,
\end{equation*}
where $\bar{\pi}\colon \P_S(\bar{\E})\to S$ is the projection map.

(3) Assume $S_2\neq S$. Then $R^1\beta_*\sEnd(\I_0^{\R_Z})=0$, $\beta_*\sEnd(\I_0^{\R_Z})$ is locally free of rank $8$, and $\det (\beta_*\sEnd(\I_0^{\R_Z})) \cong \det(\B_0)$.
\end{lemma}
\begin{proof}
(1) Since $\bar{q}|_{S_2}=0$, we have $Z'\cong \P_{S_2}(\bar{\E_2})$. The filtration (\ref{b01fil}) of $\beta^*\B_0$ also induces one for $\I_0^{\R_Z}$, and we have
\begin{equation}\label{i0fil}
    0\to \det(\R_Z) \otimes \beta^*\L^\vee \to \I_0^{\R_Z} \to \beta^*(\det(\E)\otimes (\L^\vee)^2) \to 0
\end{equation}
as well as the same for $\I_0^{\circ \R_Z}$. Recall that $\det(\R_Z) \cong \beta^*\N \otimes \O_{Z/S}(-1)$ from (\ref{detr}). Since $\bar{\pi}'\colon Z'\to S_2$ is a $\P^1$-bundle, there is a semiorthogonal decomposition
\begin{equation}\label{sodz'}
    \D(Z')=\langle \bar{\pi}'^{*}\D(S_2) \otimes \O_{Z'/S}(-1), \bar{\pi}'^{*}\D(S_2)\rangle. 
\end{equation}
This implies that
\begin{equation*}
    \Ext^1_{Z'}(\bar{\pi}'^* (\det(\E)\otimes (\L^\vee)^2)|_{S_2}, \O_{Z'/S}(-1)\otimes \bar{\pi}'^*(\N \otimes \L^\vee)|_{S_2})=0.
\end{equation*}
Thus, the sequence (\ref{i0fil}) splits after restricting to $Z'$, and we get (\ref{i0dec}).

(2) Equation (\ref{zeqhr}) implies that there is the right exact sequence
\begin{equation*}
    \O_{\P_S(\bar{\E})/S}(-2)\otimes \bar{\pi}^*\L^\vee \to \O_{\P_S(\bar{\E})} \to \O_Z\to 0
\end{equation*}
When $S_2\neq S$, we have $\bar{q}\neq 0$, which means that the kernel of the first map is torsion. Since $S$ is integral, the kernel has to be zero. 

(3) By Lemma \ref{clidealdual}, (1) we have 
\begin{align*}
    (\I_0^{\R_Z})^\vee  & \cong \I_2^{\circ \R_Z} \otimes \det(\R_Z^\vee) \otimes \beta^*(\det(\E^\vee) \otimes \L^2) \\
         & \cong \I_0^{\circ \R_Z} \otimes \det(\R_Z^\vee) \otimes \beta^*(\det(\E^\vee) \otimes \L^3).
\end{align*}
Combining it with (\ref{i0fil}) and (\ref{detr}), we have
\begin{equation}\label{endfil}
    \begin{split}
        & 0\to \I_0^{\circ \R_Z} \otimes \beta^*(\det(\E^\vee)\otimes \L^2) \to \sEnd(\I_0^{\R_Z}) \to \I_0^{\circ \R_Z}\otimes \det(\R_Z^\vee) \otimes \beta^*\L \to 0,\\
        & 0\to \O_{Z/S}(-1)\otimes \beta^*(\N\otimes \det(\E^\vee)\otimes \L) \to \I_0^{\circ \R_Z} \otimes \beta^*(\det(\E^\vee)\otimes \L^2) \to \O_Z \to 0,\\
        & 0 \to \O_Z\to \I_0^{\circ \R_Z}\otimes \det(\R_Z^\vee) \otimes \beta^*\L \to \O_{Z/S}(1)\otimes \beta^*(\N^\vee \otimes \det(\E)\otimes \L^\vee) \to 0.
    \end{split}
\end{equation}
Let $k=-1,0,1$. From the short exact sequence from (2) and the facts that $\bar{\pi}_*\O_{\P_S(\bar{\E})/S}(k-2) =0$ and $R^{\geqslant 2}\bar{\pi}_*=0$, we induce
\begin{multline*}
    0 \to \bar{\pi}_*\O_{\P_S(\bar{\E})/S}(k) \to \beta_*\O_{Z/S}(k) \to R^1 \bar{\pi}_*\O_{\P_S(\bar{\E})/S}(k-2) \otimes \L^\vee \\ \to R^1\bar{\pi}_*\O_{\P_S(\bar{\E})/S}(k) \to R^1\beta_*\O_{Z/S}(k) \to 0.
\end{multline*}
In addition, from $R^1\bar{\pi}_* \O_{\P_S(\bar{\E})/S}(k)=0$, we deduce $R^1\beta_*\O_{Z/S}(k)=0$. By Serre duality, we have
\begin{equation*}
    R^1 \bar{\pi}_*\O_{\P_S(\bar{\E})/S}(k-2) \cong (\bar{\pi}_*\O_{\P_S(\bar{\E})/S}(-k))^\vee \otimes \det(\bar{\E}).
\end{equation*}
Hence, $\beta_*\O_{Z/S}(k)$ are locally free of rank $2$ and
\begin{equation*}
    \det(\beta_*\O_{Z/S}(k))
    \cong \left\{
    \begin{array}{ll}
    \det(\bar{\E}^\vee),     &  k=1\\
    \det(\bar{\E})\otimes \L^\vee,     &  k=0\\
    \det(\bar{\E})^3\otimes (\L^\vee)^2,    & k=-1
    \end{array}
    \right..
\end{equation*}

On the other hand, we have similar long exact sequences induced by sequences (\ref{endfil}). From them, we deduce that $R^1\beta_*\sEnd(\I_0^{\R_Z})=0$ and $\beta_*\sEnd(\I_0^{\R_Z})$ is locally free of rank $8$. Note that
\begin{equation*}
    \det(\E) \cong \N \otimes \sHom(\N,\L) \otimes \det(\bar{\E}) \cong \det(\bar{\E}) \otimes \L.
\end{equation*}
Then
\begin{equation*}
    \det(\beta_*\sEnd(\I_0^{\R_Z})) \cong \det(\E)^4 \otimes (\L^\vee)^8.
\end{equation*}
Lastly, $\B_0$ has a filtration (\ref{b01fil}) with factors $\O_S, \Lambda^2\E \otimes \L^\vee, \det(\E) \otimes (\L^\vee)^2$. Since $\det(\Lambda^2 \E) \cong \det(\E)^3$, we have
\begin{align*}
    \det(\B_0) & \cong \det(\Lambda^2\E \otimes \L^\vee) \otimes \det(\E) \otimes (\L^\vee)^2\\
    & \cong \det(\E)^4 \otimes (\L^\vee)^8\\
    & \cong \det(\beta_*\sEnd(\I_0^{\R_Z})).
\end{align*}
\end{proof}

\begin{cor}\label{genb0cor}
Assume that the flat quadric surface bundle $p\colon \Q\to S$ has a smooth section and $S_2\neq S$. Then $\B_0 \cong R\beta_* \sEnd(\I_n^{\R_Z}) \cong \beta_* \sEnd(\I_n^{\R_Z})$ as sheaves of algebras for all $n\in \ZZ$.
\end{cor}
\begin{proof}
By Lemma \ref{mulisolem} (2), it suffices to prove it for $n=0$. The left $\beta^*\B_0$-module structure of $\I_0^{\R_Z}$ gives
\begin{equation}\label{endmodmap}
    f\colon \beta^*\B_0 \cong L\beta^*\B_0 \to \sEnd(\I_0^{\R_Z})
\end{equation}
and it induces
\begin{equation*}
    g\colon \B_0\to R\beta_*\sEnd(\I_0^{\R_Z}).
\end{equation*}
We only need to show that $g$ is locally an isomorphism.

Locally, we have
\begin{equation*}
    \E \cong (\N\oplus \sHom(\N,\L)) \perp \bar{\E},
\end{equation*}
where $\perp$ is the orthogonal sum of quadratic forms and the first quadratic form is given by the evaluation map $\ev_{\N}\colon \N\oplus \sHom(\N,\L)\to \L$. Denote by $\B^{\N}\cong \bigoplus_{n\in \ZZ} \B_n^{\N}$ and $\bB \cong \bigoplus_{n\in \ZZ} \bB_n$ the generalized Clifford algebras of $\ev_{\N}$ and $\bar{q}$, respectively. Then we have
\begin{equation*}
    \B_0\cong \B_0^{\N}\otimes \bB_0\oplus \B_1^{\N}\otimes \bB_1 \otimes \L^\vee\cong \bB_0\oplus \bB_0 \oplus \N\otimes \L^\vee \otimes \bB_1 \oplus \N^\vee\otimes \bB_1.
\end{equation*}
In addition, locally $\R_Z \cong \beta^*\N\oplus \O_{Z/S}(-1)$ and
\begin{equation*}
    \I_0^{\R_Z} \cong \I_0^{\O_{Z/S}(-1)} \oplus \beta^*(\N\otimes\L^\vee)\otimes \I_1^{\O_{Z/S}(-1)}.
\end{equation*}
The left $\beta^*\B_0$-module structure of $\I_0^{\R_Z}$ can be seen by writing $\B_0$ and $\I_0^{\R_Z}$ in block matrices
\begin{equation*}
    \B_0=\left(
    \begin{array}{cc}
    \bB_0     & \N^\vee\otimes \bB_1 \\
    \N\otimes \L^\vee \otimes \bB_1     & \bB_0
    \end{array}
    \right), \quad
    \I_0^{\R_Z}=\left(
    \begin{array}{c}
    \I_0^{\O_{Z/S}(-1)}\\
    \beta^*(\N\otimes\L^\vee)\otimes \I_1^{\O_{Z/S}(-1)}
    \end{array}
    \right).
\end{equation*}
By Lemma \ref{dim0lem}, we have $\bB_{n-m}\cong R\beta_*\sHom(\I_m^{\O_{Z/S}(-1)}, \I_n^{\O_{Z/S}(-1)})$. This implies that $g$ is an isomorphism. 
\end{proof}

\begin{remark}
The corollary has an easier proof if $S_2\subset S$ has codimension at least $2$ or $S$ is proper and integral. The proof goes as follows.

By Lemma \ref{genb0lem} (3), we have $R\beta_*\sEnd(\I_0^{\R_Z}) \cong \beta_*\sEnd(\I_0^{\R_Z})$ is locally free and $\det(\B_0) \cong \det(\beta_*\sEnd(\I_0^{\R_Z}))$. Hence, $\det(g)\in \Gamma(S,\O_S)$. By Lemma \ref{sdb0lem} (2), the map $g$ induced in the proof of the previous corollary is an isomorphism on $S\backslash S_2$ and thus $\det(g)|_{S\backslash S_2}$ is a unit. If $S$ is proper and integral, then $\det(g)$ is a non-zero constant and $g$ is an isomorphism on $S$. On the other hand, $\{\det(g)=0\}\subset S$ is either empty or has codimension $1$. Then $g$ is an isomorphism on $S$ if $S_2\subset S$ has codimension at least $2$.
\end{remark}

We give an explicit description of the map $\beta\colon Z\to S$ for the universal quadric surface bundle with a smooth section.

\begin{ex}\label{univqsb}
The universal family $\Q$ of quadric surface bundles with a smooth section is parametrized by $S \cong \AA^3 \cong \Spec(\kk[a,b,c])$ and the quadratic form is
\begin{equation*}
    q(x)=x_1x_2+  ax_3^2 +bx_3x_4 +c x_4^2.
\end{equation*}
The smooth section is given by $\{x_2=x_3=x_4=0\}$ or $\{x_1=x_3=x_4=0\}$. The hyperbolic reduction with respect to either smooth section is $\bar{q}=ax_3^2 +bx_3x_4 +c x_4^2$. Then 
\begin{equation*}
    Z \cong \Proj (\frac{\kk[a,b,c,x_3, x_4]}{ax_3^2 +bx_3x_4 +c x_4^2}), \quad \wt{S} \cong \Spec(\frac{\kk[a,b,c,d]}{d^2-b^2+4ac}),
\end{equation*}
where $Z$ is the hyperbolic reduction as well as the scheme parametrizing lines that intersect the smooth section, and $\wt{S}$ is the double cover over $\AA^3$. Let $\sigma$ be the involution of the double cover $\wt{S}\to \AA^3$. Then there is a factorization
\begin{equation*}
    \beta: Z\xrightarrow{h} \wt{S} \to S \cong \AA^3
\end{equation*}
where $h$ and $\sigma \circ h$ are the two minimal resolutions of the affine nodal quadric threefold $\wt{S}$. In addition, $S_2\cong \{a=b=c=0\}$ is the origin and has codimension $3$ in $S\cong \AA^3$.
\end{ex}
\section{Quadric surface bundles with a smooth section}\label{smoothsec}
Let $p\colon \Q\to S$ be a flat quadric surface bundle with a smooth section and let $\A_{\Q}$ be its residual category. In this section, we prove that $\A_{\Q}$ is geometric. We give two proofs where the easier proof is described in Theorem \ref{main1} and the harder proof in Theorem \ref{main2}. The harder proof in addition gives an explicit description of the Fourier-Mukai kernels of the embedding functors $\Psi_n\colon \A_{\Q}\to \D(\Q)$. 

Firstly, we provide a type of mutations for derived categories of not necessarily smooth schemes. Recall that a morphism is \textit{perfect} if it is pseudo-coherent and has finite Tor-dimension. A proper \textit{local complete intersection} morphism (locally factors as a Koszul-regular closed immersion followed by a smooth morphism) is perfect and it has \textit{invertible} relative dualizing complex (a degree shift of a line bundle); see Example 3.2 in \cite{jquot-1}.
\begin{lemma}\label{mulem}
Let $f\colon X\to S$ be a proper and perfect morphism of noetherian schemes. Denote by $\omega_f=f^!(\O_S)$ the relative dualizing complex. Assume $\omega_f$ is invertible, e.g., when $f$ is a Gorenstein or a proper local complete intersection morphism. Denote by $S_{X/S}= -\otimes \omega_f$ the equivalence functor on $\D(X)$ (This is the relative Serre functor on $\DD^{\perf}(X)$). Then $f^! \cong S_{X/S}\circ Lf^*$ on $\D(X)$. 

(1) Assume there is an $S$-linear semiorthogonal decomposition
\begin{equation}\label{ssod}
    \D(X) =\langle \A_1, \A_2 \rangle
\end{equation}
where $S$-linear means that $\F\otimes f^*\G \in \A_i$ for every $\F\in \A_i$, $\G\in \DD^{\perf}(S)$, the derived category of perfect complexes, and $i=1,2$. Assume that $\A_1$ or $\A_2$ is equivalent to $Lf^*\D(S)\otimes T$ where $T\in \DD^{\perf}(X)$ is relative exceptional, i.e., $Rf_*R\sHom_{\O_X}(T, T)\cong \O_S$. Then both $\A_1$ and $\A_2$ are admissible, i.e., inclusion functors have both left and right adjoints, and there is an $S$-linear semiorthogonal decomposition
\begin{equation}\label{ssodmu}
    \D(X) =\langle \A_2\otimes \omega_f, \A_1 \rangle.
\end{equation}

(2) More generally, assume that there is an $S$-linear semiorthogonal decomposition
\begin{equation*}
    \D(X) =\langle \A_1, \dots, \A_n \rangle.
\end{equation*}
Assume that there exists some $i_0$ such that for all $i\neq i_0$, we have $\A_i\cong Lf^*\D(S)\otimes T_i$ for relative exceptional $T_i\in \DD^{\perf}(X)$. Then each $\A_i, 1\leqslant i\leqslant n$ is admissible and there is an $S$-linear semiorthogonal decomposition
\begin{equation}\label{nssodmu}
    \D(X) =\langle \A_n\otimes \omega_f, \A_1 \dots, \A_{n-1} \rangle.
\end{equation}
\end{lemma}
\begin{proof}
The claim that $f^! \cong S_{X/S}\circ Lf^*$ is given in Theorem 3.1 (2) of \cite{jquot-1}.

(1) We will make use of Lemma 2.7 in \cite{kuzbc}, which states that $\A_1, \A_2$ is semiorthogonal, i.e., $\Hom_{\D(X)}(\F_2, \F_1)=0$ for \textbf{all} $\F_i\in \A_i$, $i=1,2$ if and only if $Rf_* R\sHom_{\O_X}(\F_2, \F_1)=0$ for \textbf{all} $\F_i\in \A_i$, $i=1,2$.

We assume that $\A_1\cong Lf^*\D(S)\otimes T$ for $T\in \DD^{\perf}(X)$ relative exceptional. The other case is similar. Recall that for a full subcategory $\A\subset \D(X)$, the \textit{right orthogonal} to $\A$ is
\begin{equation*}
    \A^{\perp} = \{\F\in \D(X)\,|\, \Hom_{\D(X)}(\cK,\F)=0, \forall \cK\in \A\}
\end{equation*}
and the \textit{left orthogonal} to $\A$ is
\begin{equation*}
    {}^{\perp}\A = \{\F\in \D(X)\,|\, \Hom_{\D(X)}(\F, \cK)=0, \forall \cK\in \A\}.
\end{equation*}
From the SOD (\ref{ssod}), we get $\A_2\cong {}^{\perp}\A_1$ and $\A_1\cong \A_2^{\perp}$. The existence of the SOD (\ref{ssodmu}) is equivalent to $S_{X/S}(\A_2) \cong \A_1^{\perp}$.

Firstly, we show $S_{X/S}(\A_2) \subset \A_1^{\perp}$. Let $\F\in \A_2$, $\G\in \D(S)$. Then
\begin{equation*}
    \begin{split}
        Rf_*R\sHom_{\O_X}(Lf^*\G\otimes T, S_{X/S}(\F)) & \cong Rf_*R\sHom_{\O_X}(Lf^*\G, T^\vee \otimes S_{X/S}(\F))\\
        & \cong R\sHom_{\O_S}(\G, Rf_*(T^\vee \otimes S_{X/S}(\F))).
    \end{split}
\end{equation*}
Note that $T$ is a perfect complex. The local Grothendieck-Serre duality gives
\begin{equation*}
    Rf_*(T^\vee \otimes S_{X/S}(\F)) \cong Rf_* R\sHom_{\O_X}(T, S_{X/S}(\F)) \cong R\sHom_{\O_S}(Rf_* R\sHom_{\O_X}(\F, T), \O_S).
\end{equation*}
Since $\F\in \A_2\cong {}^{\perp}\A_1$, we have $Rf_* R\sHom_{\O_X}(\F, T)=0$. Hence, $S_{X/S}(\A_2) \subset \A_1^{\perp}$.

Next, we show $S_{X/S}^{-1}(\A_1^{\perp})\subset \A_2 \cong {}^{\perp}\A_1$. Let $\F\in \A_1^{\perp}$ and $\G\in \D(S)$. The local Grothendieck-Verdier duality gives
\begin{equation*}
    \begin{split}
        Rf_*R\sHom_{\O_X}(S_{X/S}^{-1}(\F), Lf^*\G\otimes T) & \cong Rf_*R\sHom_{\O_X}(\F, S_{X/S}(Lf^*\G\otimes T))\\
        & \cong Rf_*R\sHom_{\O_X}(\F, (f^!\G)\otimes T)\\
        & \cong R\sHom_{\O_S}(Rf_*(\F\otimes T^\vee), \G).
    \end{split}
\end{equation*}
Since $\F\in \A_1^{\perp}$, we have $Rf_*(\F\otimes T^\vee)=0$. Hence, $S_{X/S}^{-1}(\A_1^{\perp})\subset \A_2$. This concludes $S_{X/S}(\A_2) \cong \A_1^{\perp}$.

It remains to show the admissibility of $\A_i, i=1,2$. From the SOD (\ref{ssod}), we get that $\A_1$ is left admissible and $\A_2$ is right admissible. From the SOD (\ref{ssodmu}), we get that $\A_1$ is right admissible. The SOD (\ref{ssodmu}) induces a SOD $\langle \A_2, \A_1\otimes \omega_f^{-1}\rangle$ and thus $\A_2$ is also left admissible.

(2) There are two cases. Let $\A=\langle \A_1, \dots, \A_{n-1}\rangle$ when $i_0\neq n$. Then we can apply (1) to get the SOD (\ref{nssodmu}). When $i_0=n$, we can apply (1) to get a SOD
\begin{equation*}
    \D(X)=\langle \A_2, \dots, \A_n, \A_1\otimes \omega_f^{-1} \rangle.
\end{equation*}
We apply (1) for $n-2$ more times and get
\begin{equation*}
    \begin{split}
        \D(X) & =\langle \A_n, \A_1\otimes \omega_f^{-1}, \dots,  \A_{n-1}\otimes \omega_f^{-1} \rangle\\
        & \cong \langle \A_n\otimes \omega_f, \A_1, \dots,  \A_{n-1}\rangle.
    \end{split}
\end{equation*}
Each $\A_i, 1\leqslant i\leqslant n$ is admissible because we can construct SODs similar to (1) with $\A_i$ being the leftmost or the rightmost component.
\end{proof}

\begin{theorem}\label{main1}
Let $p\colon \Q\to S$ be a flat quadric surface bundle where $S$ is an integral noetherian scheme over $\kk$ with $\cchar(\kk)\neq 2$. Assume that $p$ has a smooth section and the second degeneration $S_2\neq S$. Then there is a semiorthogonal decomposition
\begin{equation}
    \D(\Q) =\langle \D(\bar{\Q}), p^*\D(S), p^*\D(S) \otimes \O_{\Q/S}(1) \rangle
\end{equation}
where $\bar{\Q}$ is the hyperbolic reduction of $\Q$ with respect to the smooth section in Definition \ref{hypreddef}.
\end{theorem}
\begin{proof}
Let $\N$ be the regular isotropic line bundle corresponding to the smooth section of $p$. Remark 2.6 of \cite{ksgroring} provides the following picture:
\begin{equation}\label{hrdiag}
    \begin{tikzcd} 
    E \arrow[hook]{r}{\bar{j}} \arrow{d}[swap]{\bar{f}} & \Q' \arrow{d}[swap]{f} \arrow{rd}{g} & & &\\
    \P_S(\N) \arrow[hook]{r}{j} & \Q  & \P_S(\E/\N) \arrow[hookleftarrow]{r} & \P_S(\N^\perp/\N) \arrow[hookleftarrow]{r} & \bar{\Q}
    \end{tikzcd}
\end{equation}
where $g\circ f^{-1}\colon \Q\dashrightarrow \P_S(\E/\N)$ is the relative linear projection from $\P_S(\N) \cong S$, the scheme $\bar{\Q}$ is the hyperbolic reduction with respect to $\N$,
\begin{equation*}
    \Q' \cong \Bl_{\P_S(\N)}(\Q) \cong \Bl_{\bar{\Q}}(\P_S(\E/\N)),
\end{equation*}
the subscheme $E \cong \P_S(\N^\perp/\N) \subset \Q'$ is the exceptional locus of $f$, the map $\bar{f}=f|_E$, and maps $j, \bar{j}$ are inclusions.

Let $D\subset \Q'$ be the exceptional locus of $g$. Let $H$ and $h$ be the relative hyperplane classes of $\Q$ and $\P_S(\E/\N)$, respectively. Use the same notations for the pull-back classes on $\Q'$. In $\Pic(\Q')/f^*p^*\Pic(S)$, there are relations
\begin{equation}\label{relofdiv}
    \left\{
    \begin{array}{l}
    H = h+E \\
    h = D+E
    \end{array}
    \right..
\end{equation}

Let $\bar{q}\colon \N^\perp/\N \to \L$ be the hyperbolic reduction and recall $\bar{\Q}=\{\bar{q}=0\} \subset \P_S(\N^\perp/\N)$. Since $S_2\neq S$ and $S$ is integral, we have $\bar{q}\neq 0$ and $\bar{\Q}\hookrightarrow \P_S(\E/\N)$ is a regular immersion of codimension $2$. Locally, $\P_S(\N)\subset \Q$ is defined by $\{xy+\bar{q}=0\}$, where $x,y$ are variables for $\N$ and $\sHom(\N,\L)$, respectively. Restricting to $\{x\neq 0\}$, the smooth section is defined by $\{z=w=0\}$, where $z,w$ are variables for $\N^\perp/\N$. Thus, $\P_S(\N)\hookrightarrow \Q$ is also a regular immersion of codimension $2$. The blow up formulas for derived categories in Theorem 6.11 of \cite{jquot-1} can be applied to maps $f, g$. For the blow-up map $f$, there is a semiorthogonal decomposition
\begin{equation}\label{blowupsod1}
        \D(\Q') = \langle Lf^*\D(\Q), \bar{j}_*L\bar{f}^*\D(\P_S(\N)) \rangle \cong \langle \D(\Q), \D(S) \otimes \bar{j}_*\O_E \rangle,
\end{equation}
where $\D(S) \otimes \bar{j}_*\O_E$ denotes a subcategory that is equivalent to $\D(S)$, and it is obtained by pulling back objects in $\D(S)$ via the map $f'^*$ followed by tensoring with $\bar{j}_*\O_E$. Here $f'=p\circ f\colon \Q'\to S$ is the composition map and $f'$ is flat. The equivalence of the second components in the SOD (\ref{blowupsod1}) comes from the observation that $p\circ j\colon \P_S(\N)\to S$ is an isomorphism and we have
\begin{equation*}
    \bar{f}\cong p\circ j\circ\bar{f} \cong p\circ f\circ \bar{j} \cong f'\circ \bar{j}.
\end{equation*}
Similarly for the blow-up map $g$, there is a semiorthogonal decomposition
\begin{equation}\label{blowupsod2}
    \D(\Q') = \langle \D(\bar{\Q}), \D(\P_S(\E/\N)) \rangle.
\end{equation}

There is a semiorthogonal decomposition
\begin{equation*}
    \D(\Q)= \langle \A_{\Q}, p^*\D(S), p^*\D(S)\otimes \O(H)\rangle \cong \langle \A_{\Q}, \D(S)\otimes \O_{\Q}, \D(S)\otimes \O(H)\rangle
\end{equation*}
where $\A_{\Q}$ is the residual category and for the equivalence of components, similar notations are used as in the SOD (\ref{blowupsod1}). Therefore, the SOD (\ref{blowupsod1}) can be expanded as 
\begin{align}
    \D(\Q') & = \langle \A_{\Q}, \D(S)\otimes \O_{\Q'}, \D(S)\otimes \O(H), \D(S)\otimes \bar{j}_*\O_E\rangle \\
    & \cong \langle \A_{\Q}, \D(S)\otimes \O_{\Q'}, \D(S)\otimes \O(H-E), \D(S)\otimes \O(H)\rangle \label{mut1}\\
    & \cong \langle \D(S)\otimes \O(H-3h+D), \A_{\Q}, \D(S)\otimes \O_{\Q'}, \D(S)\otimes \O(h) \rangle \label{mut2}\\
    & \cong \langle \A_{\Q}, \D(S)\otimes \O(-h), \D(S)\otimes \O_{\Q'}, \D(S)\otimes \O(h) \rangle \label{mut3}\\
    & \cong \langle \A_{\Q}, \D(\P_S(\E/\N)) \rangle. \label{mut4}
\end{align}
The equivalences above are obtained by mutations. Denote by $\LL_T$ and $\RR_T$, respectively, the left and right mutation functors through $\D(S)\otimes T$ when $T$ is a relative exceptional object over $S$ (cf. \S 3.11 in \cite{jquot-1}). Note that up to divisors of $S$, we have $H-E=h$, $H-3h+D=-h$ by relations (\ref{relofdiv}) and the relative canonical divisor $K_{\Q'/S}=-3h+D$. In addition, up to a degree shift the relative dualizing complex $\omega_{f'}\cong \O(K_{\Q'/S})$. (\ref{mut1}) applies $\LL_{\O(H)}$ to $\D(S) \otimes  \bar{j}_*\O_E$; (\ref{mut2}) applies $-\otimes \omega_{f'}$ to $\D(S)\otimes \O(H)$, which uses Lemma \ref{mulem} (2); and (\ref{mut3}) applies $\LL_{\O(-h)}$ to $\A_{\Q}$. Comparing SODs (\ref{mut4}) and (\ref{blowupsod2}), we have $\A_{\Q}\cong \D(\bar{\Q})$.
\end{proof}

Since $\A_{\Q}\cong \D(\bar{\Q})$ is obtained by mutations, we lose information on the embedding functor $\A_{\Q}\to \D(\Q)$. To remedy this problem, we focus on working with $\A_{\Q}$ by itself and show that $\A_{\Q}$ is geometric using Corollary \ref{genb0cor} and the known description of $\A_{\Q}$ below.

Theorem 4.2 in \cite{kuzqfib} ($\kk$ algebraically closed and $\cchar(\kk)=0$) and Theorem 2.2.1 in \cite{abbqfib} (arbitrary field $\kk$) state that for all $n\in\ZZ$, there are semiorthogonal decompositions
\begin{equation}\label{kuzsod}
    \D(\Q) =\langle \Phi_n(\D(S,\B_0)), p^*\D(S)\otimes \O_{\Q/S}(1), p^*\D(S) \otimes \O_{\Q/S}(2) \rangle,
\end{equation}
where 
\begin{equation}\label{phin}
    \Phi_n\colon \D(S,\B_0) \to \D(\Q), \quad \F \mapsto p^*(\F)\otimes_{p^*\B_0}^{\LL} \cK_n
\end{equation}
are fully faithful functors of Fourier-Mukai type with kernels $\cK_n$. The kernels $\cK_n$ are left $\B_0$-modules constructed by
\begin{equation}\label{fmkseq}
    0\to \O_{\P_S(\E)/S}(-1)\otimes \pi^*\B_{n-1} \xrightarrow{\phi_n^\circ} \pi^*\B_{n} \to i_* \cK_n \to 0,
\end{equation}
where $\pi\colon \P_S(\E)\to S$ is the projetion, $i\colon \Q\hookrightarrow \P_S(\E)$ is the embedding and $\phi_n^\circ$ is the map defined in Remark \ref{spinordef2}. We can regard $\cK_n$ as the $n$-th spinor sheaf with respect to the zero isotropic subbundle.

We use the same notations as in Section \ref{hilbsec} and consider $Z\subset M$ the subscheme parametrizing lines in the fibers of $p\colon \Q\to S$ that intersect the smooth section. Note that $\beta\colon Z\to S$ together with $f\colon \beta^*\B_0\to \sEnd(\I_0^{\R_Z})$ from (\ref{endmodmap}) give a morphism
\begin{equation}\label{ncmor}
    \gamma=(\beta,f)\colon (Z, \sEnd(\I_0^{\R_Z})) \to (S,\B_0)
\end{equation}
of non-commmutative schemes as in Definition \ref{ncdef}.
\begin{prop}\label{nceqprop}
$R\gamma_*\colon  \DD^*(Z,\sEnd(\I_0^{\R_Z}))\to \DD^*(S,\B_0)$ is an equivalence for $*=-, \textrm{b}$.
\end{prop}
\begin{proof}
Consider functors
\begin{equation*}
    R\gamma_*\colon  \DD^-(Z,\sEnd(\I_0^{\R_Z}))\to \DD^-(S,\B_0), \quad L\gamma^*\colon  \DD^-(S,\B_0) \to \DD^-(Z,\sEnd(\I_0^{\R_Z}))
\end{equation*}
defined in Appendix \ref{ncschsec}. Firstly, we show $ R\gamma_*$ and $L\gamma^*$ are inverse functors.  Corollary \ref{genb0cor} indicates that $\B_0\cong R\gamma_*\sEnd(\I_0^{\R_Z})$. Then $R\gamma_* L\gamma^*$ is identity by the projection formula (Proposition \ref{projf}).

Conversely, we claim that $R\gamma_*\F=0$ implies $\F=0$ for all $\F\in\DD^-(Z, \sEnd(\I_0^{\R_Z}))$. Denote by $\cH^i$ the $i$-th cohomology sheaf. Using the spectral sequence $R^i\gamma_*\cH^j(\F)\Rightarrow R^{i+j}\gamma_* \F$, we can reduce to the case that $\F\in \Coh(Z, \sEnd(\I_0^{\R_Z}))$. Adopt the notations in Lemma \ref{genb0lem} (1) and let $j\colon Z'\hookrightarrow Z$ be the inclusion. Since $\beta|_{Z\backslash Z'}$ is finite, the condition $R\gamma_*\F=0$ implies that $\F$ is supported on $Z'\cong \P_{S_2}(\bar{\E_2})$. If additionally we can prove that $j^*\F=0$, then $\F=0$. Let $\F_1= j^*\F \in \Coh(Z',\sEnd(\I_0^{\R_{Z'}}))$. Since there is an equivalence
\begin{equation*}
    -\otimes (\I_0^{\R_{Z'}})^\vee\colon \Coh(Z')\xrightarrow{\cong} \Coh(Z',\sEnd(\I_0^{\R_{Z'}})),
\end{equation*}
we have $\F_1 \cong \G \otimes (\I_0^{\R_{Z'}})^\vee$ for some $\G\in \Coh(Z')$. From the decomposition (\ref{i0dec}), we deduce that
\begin{equation*}
    \G \otimes (\I_0^{\R_{Z'}})^\vee \cong \G\otimes \bar{\pi}'^* \left( (\det(\E^\vee)\otimes \L^2)|_{S_2}\right) \oplus \G(1) \otimes \bar{\pi}'^*\left( (\N^\vee \otimes \L)|_{S_2}\right).
\end{equation*}
Note that $\beta\colon Z\to S$ is a proper morphism with fibers of dimension at most $1$ and $\bar{\pi}'\colon Z'\to S_2$ is the base change of $\beta$ along the inclusion $S_2\hookrightarrow S$. Together with $R\gamma_*\F\cong R\beta_* \F=0$, Proposition 2.11 in \cite{bbflop} deduces that $R\bar{\pi}'_*\F_1=0$. This means that $R\bar{\pi}'_*\G=0$ and $R\bar{\pi}'_*\G(1)=0$, or equivalently
\begin{equation*}
    \Hom_{\D(Z')}(\bar{\pi}'^*\cK, \G)=0, \quad \Hom_{\D(Z')}((\bar{\pi}'^*\cK)(-1), \G)=0
\end{equation*}
for all $\cK\in \D(Z')$. The SOD (\ref{sodz'})
\begin{equation*}
    \D(Z')=\langle \bar{\pi}'^{*}\D(S_2)\otimes \O_{Z'/S}(-1), \bar{\pi}'^{*}\D(S_2)\rangle 
\end{equation*}
implies that $\G=0$. Thus, $\F_1=0$ and $\F=0$.

For every $\F\in\DD^-(Z,\sEnd(\I_0^{\R_Z}))$, consider the exact triangle
\begin{equation*}
    L\gamma^* R\gamma_* \F \to \F \to \cK
\end{equation*}
where $\cK$ is the cone of the first map. Applying $R\gamma_*$ to the exact triangle, we get $R\gamma_*\cK=0$. Hence, $\cK=0$ and $L\gamma^* R\gamma_*$ is the identity.

The equivalence on $\DD^-$ implies that
\begin{equation*}
    R\gamma_*\colon  \D(Z,\sEnd(\I_0^{\R_Z}))\to \D(S,\B_0)
\end{equation*}
is fully faithful. Recall that $R\gamma_* L\gamma^*\cong \id$. From the proofs of Lemma 2.4 and Corollary 2.5 in \cite{kuzdp6}, we get that $R\gamma_*$ is also essentially surjective on $\D$. Thus, $R\gamma_*$ is an equivalence on $\D$.
\end{proof}

\begin{theorem}\label{main2}
Let $p\colon \Q\to S$ be a flat quadric surface bundle where $S$ is an integral noetherian scheme over $\kk$ with $\cchar(\kk)\neq 2$. Assume that $p$ has a smooth section and the second degeneration $S_2\neq S$. Then for all $n\in \ZZ$ there are semiorthogonal decompositions
\begin{equation}
    \D(\Q) =\langle \Psi_n(\D(Z)), p^*\D(S)\otimes \O_{\Q/S}(1), p^*\D(S) \otimes \O_{\Q/S}(2) \rangle,
\end{equation}
where $Z$ is the scheme over $S$ parametrizing lines in the fibers of $p$ that intersect the smooth section and $\Psi_n\colon \D(Z) \to \D(\Q)$ are fully faithful functors of Fourier-Mukai type. In addition, $Z$ is isomorphic to the hyperbolic reduction $\bar{\Q}$ in Theorem \ref{main1}.

More specifically, let $\P_Z(\R_Z) \subset \Q \times_S Z$ be the universal family of lines that $Z$ parametrizes. Then the kernel of $\Psi_n$ is $\cS_n^{\R_Z}$, the $n$-th spinor sheaf with respect to the isotropic subbundle $\R_Z$ in Definition \ref{spinordef}. 
\end{theorem}
\begin{proof}
For simplicity, denote by $\otimes, g_*, g^*$ the derived tensor product, the derived pull-forward and pull-back functors of a map $g$, respectively. The isomorphism $Z\cong \bar{Q}$ has been pointed out by (\ref{zeqhr}).

Consider the composition
\begin{equation}\label{psin}
    \Psi_n\colon \D(Z) \xrightarrow[\cong]{-\otimes \I_0^{\circ \R_Z}} \D(Z, \sEnd(\I_0^{\R_Z})) \xrightarrow[\cong]{\gamma_*} \D(S,\B_0) \xrightarrow{\Phi_n} \D(\Q),
\end{equation}
where $\Phi_n$ is defined in (\ref{phin}) and $\gamma=(\beta,f)\colon (Z, \sEnd(\I_0^{\R_Z})) \to (S,\B_0)$ in (\ref{ncmor}). From Lemma \ref{clidealdual}, we get
\begin{equation*}
    \I_0^{\circ \R_Z} \cong (\I_0^{\R_Z})^\vee \otimes \det(\R_Z)\otimes \beta^*(\det(\E)\otimes (\L^\vee)^3).
\end{equation*}
Thus, the first functor in (\ref{psin}) is an equivalence because $-\otimes (\I_0^{\R_Z})^\vee$ is such. From Proposition \ref{nceqprop}, we get $\gamma_*$ is also an equivalence. Thus, $\Psi_n$ is fully faithful and for every $\F\in\D(Z)$,
\begin{equation*}
    \Psi_n(\F) = p^*\beta_*(\F\otimes \I_0^{\circ \R_Z})\otimes_{p^*\B_0} \cK_n.
\end{equation*}

Consider the Cartesian squares
\begin{equation*}
    \begin{tikzcd}
    \Q\times_S Z \arrow[hook]{r}{i_Z} \arrow{d}{\beta_{\Q}} & \P_S(\E)\times_S Z \arrow{r}{\pi_Z} \arrow{d}{\beta_{\E}} & Z\arrow{d}{\beta}\\
    \Q \arrow[hook]{r}{i} & \P_S(\E) \arrow{r}{\pi} & S.
    \end{tikzcd}
\end{equation*}
Recall that $p=\pi\circ i$ and denote $p_Z=\pi_{Z}\circ i_Z$. Let 
$\theta=\beta_{\E}\circ \pi = \pi_Z \circ \beta$. Since $\pi$ is flat, the right Cartesian square is exact. We have 
\begin{align*}
    i_*\Psi_n(\F) & \cong \pi^*\beta_*(\F\otimes \I_0^{\circ \R_Z})\otimes_{\pi^*\B_0} i_*\cK_n\\
    & \cong \beta_{\E*}\pi_Z^*(\F\otimes \I_0^{\circ \R_Z})\otimes_{\pi^*\B_0} i_*\cK_n\\
    & \cong \beta_{\E*} \left( \pi_Z^*(\F\otimes \I_0^{\circ \R_Z})\otimes_{\theta^*\B_0} \beta_{\E}^*i_*\cK_n \right).
\end{align*}
On the other hand, applying $\pi_Z^*\I_0^{\circ \R_Z}\otimes_{\theta^*\B_0} \beta_{\E}^*(-)$ to the sequence (\ref{fmkseq}) and using the isomorphism from Lemma \ref{mulisolem} (1), we have
\begin{equation*}
    0 \to \beta_{\E}^*\O_{\P_S(\E)/S}(-1)\otimes \pi_Z^*\I_{n-1}^{\circ \R_Z} \xrightarrow{\phi_n^{\circ}} \pi_Z^* \I_{n}^{\circ \R_Z} \to \pi_Z^*(\I_0^{\circ \R_Z})\otimes_{\theta^*\B_0}\beta_{\E}^* (i_* \cK_n) \to 0.
\end{equation*}
From Definition \ref{spinordef}, we get that the last term is $i_{Z*}\cS_n^{\R_Z}$. Then we have
\begin{align*}
    i_*\Psi_n(\F) & \cong \beta_{\E*}(\pi_Z^*\F \otimes i_{Z*}\cS_n^{\R_Z})\\
    & \cong \beta_{\E*}i_{Z*}(p_Z^*\F \otimes \cS_n^{\R_Z})\\
    & \cong i_*\beta_{\Q*}(p_Z^*\F \otimes \cS_n^{\R_Z}).
\end{align*}
Hence, $\Psi_n$ are Fourier-Mukai functors with kernels $\cS_n^{\R_Z}$.
\end{proof}
\section{Quadric surface bundles over surfaces}\label{surfbsec}
Use the same notations as in Section \ref{hilbsec}. Throughout this section, assume that $\kk$ is algebraically closed and $p\colon \Q\to S$ is a flat quadric surface bundle (except for Lemma \ref{bealem}) where $\Q$ is smooth and $S$ is a smooth surface. Recall that $S_k\subset S$ is the $k$-th degeneration locus. We will show that the residual category $\A_{\Q}$ is twisted geometric by performing birational transformation to the relative Hilbert scheme of lines following the approach of \cite{kuzline}.

Recall that the relative Hilbert scheme of lines $\rho \colon M\to S$ factors as $M\xrightarrow{\tau} \wt{S} \xrightarrow{\alpha} S$. From Lemma \ref{disccover} and Lemma \ref{bealem}, we get that the discriminant cover $\alpha$ in this section is a double cover ramified along $S_1$. When $p\colon \Q\to S$ has simple degeneration, the map $\tau\colon M\to \wt{S}$ is a smooth conic bundle and the Azumaya algebra $\wt{\A}$ corresponding to $\tau$ satisfies $\alpha_*\wt{\A}\cong \B_0$. This implies that $\A_{\Q}$ is equivalent to the twisted derived category $\D(\wt{S}, \wt{\A})$. But when $p$ does not have simple degeneration, $\tau$ is no longer a smooth conic bundle. The main idea of the section is to modify $\tau$ so that we can still describe $\A_{\Q}$ as a twisted derived category. In the first part, we will perform birational transformations to $\tau$ and obtain a smooth conic bundle $\tau_+\colon M^+\to S^+$ in Proposition \ref{birtran}. The second part is devoted to the proof of Theorem \ref{main3}, where we show $\A_{\Q}\cong \D(S^+, \A^+)$ for $\A^+$ the Azumaya algebra corresponding to $\tau_+$.

\medskip

When the total space of a quadric bundle (not necessarily a quadric surface bundle) is smooth and the base is a smooth surface, fibers of the quadric bundle will not be too singular:

\begin{lemma}[{\cite[Proposition 1.2 (iii)]{beaprym}}] \label{bealem}
Let $p\colon \Q\to S$ be a flat quadric bundle where $\Q$ is smooth and $S$ is a smooth surface. Then $S_3=\emptyset$, $S_1\subset S$ is a curve with at most a finite number of ordinary double points, and the singular locus of $S_1$ is $S_2$.
\end{lemma}

Since $S_3=\emptyset$, for every geometric point $s\in S_2$, the geometric fiber $M_s = \Sigma_s^+\cup \Sigma_s^-$ is the union of two planes $\Sigma_s^{\pm}\cong \P^2$ intersecting at a point. Denote by $m_s=\Sigma_s^+\cap \Sigma_s^-$ the intersection point. The geometry of $M$ is described below.

\begin{lemma}\label{mgeom}
Let $p\colon \Q\to S$ be a flat quadric surface bundle where $\Q$ is smooth and $S$ is a smooth surface. Then the relative Hilbert scheme of lines $M$ has at most a finite number of ordinary double points $\{m_s\}_{s\in S_2}$.
\end{lemma}
\begin{proof}
By Lemma \ref{bealem}, $S_2$ is a finite set of points and thus so is $\{m_s\}_{s\in S_2}$. For a point $x\in S$, denote by $K_x\subset \E_x$ the kernel of the quadratic form $q_x$ over the residue field of the point $x$. A point of $M$ is represented by $(x, K)$ where $x$ is a point of $S$ and $K$ is a $2$-dimensional subspace of $\E_x$. Proposition 2.1 in \cite{kuzline} states that $M$ is smooth at $(x,K)$ if $\dim(K\cap K_x)\leqslant 1$. Hence, $M$ is smooth away from $\{m_s\}_{s\in S_2}$.

For the singularity of $M$, we can replace $S$ by $\Spec(\widehat{\O_{S,s}})$ where $\widehat{\O_{S,s}}\cong \kk \llbracket t_1, t_2 \rrbracket$ is the formal completion at a point $s\in S_2$. Note that all units in $\kk \llbracket t_1, t_2 \rrbracket$ are squares. By Corollary 3.3 in \cite{baeqf}, the quadratic form can be written as
\begin{equation*}
    q=ax_1^2+bx_1x_2+cx_2^2 +x_3^2+x_4^2 
\end{equation*}
where $a,b,c$ are contained in the ideal $(t_1,t_2)$. Recall that $M\subset S\times \Gr(2,4)$ is the zero locus of the section $s_M$ in (\ref{msec}). We can refer to Section 3.3 in \cite{hvav} for the explicit correspondence. Denote the variables of $\Gr(2,4)$ by $\{y_{ij}\}_{1\leqslant i<j\leqslant 4}$ and assume that the singular point of $M$ is $(0, [1:0:\dots :0])$. Consider the open neighborhood $U:=S\times \{y_{12}=1\}$ of the singular point. Then $M\cap U$ is defined by
\begin{equation}\label{meq1}
    \left\{
    \begin{array}{l}
    y_{23}^2+ y_{24}^2=-2a\\
    y_{13}y_{23}+y_{14}y_{24}=b\\
    y_{13}^2+y_{14}^2=-2c
    \end{array}
    \right..
\end{equation}
By Lemma \ref{bealem}, $S_1=\{\det(b_q)=4ac-b^2=0\}$ has an ordinary double point at $0\in S$. The degree $2$ term of $\det(b_q)$ is $4l_al_c-l_b^2$ where $l_a, l_b, l_c$ are the linear terms of $a,b,c$, respectively. Let $\Delta$ be the discriminant of $4l_al_c-l_b^2$. From $\Delta\neq 0$, we get that $l_a, l_b, l_c$ are not proportional. We can assume $a=t_1, c=t_2$ and $l_b=\lambda t_1+ \mu t_2$ for $\lambda,\mu \in\kk$. Then $\Delta=16(1-\lambda\mu)\neq 0$ and Equations (\ref{meq1}) reduce to 
\begin{equation}\label{meq2}
    \left\{ \frac{\lambda}{2}y_{23}^2 +y_{13}y_{23} +\frac{\mu}{2}y_{13}^2 +\frac{\lambda}{2} y_{24}^2 +y_{14}y_{24} +\frac{\mu}{2}y_{14}^2 +\epsilon=0 \right\} \subset \AA_{\kk}^4,
\end{equation}
where $\epsilon$ is a sum of degree $\geqslant 3$ terms in $y_{13}, y_{23}, y_{14}, y_{24}$. Since $1-\lambda\mu\neq 0$, the discriminants of 
\begin{equation*}
    \frac{\lambda}{2}y_{2i}^2 +y_{1i}y_{2i} +\frac{\mu}{2}y_{1i}^2, \quad i=3,4
\end{equation*}
are non-zero. Hence, the degree $2$ part of (\ref{meq2}) has full rank and $M$ has an ordinary double point.
\end{proof}

We will construct a minimal resolution $\wt{M}$ of the nodal $3$-fold $M$. This is obtained by blowing up $M$ along $\bigsqcup_{s\in S_2} \Sigma_s^+$ (or alternatively $\bigsqcup_{s\in S_2} \Sigma_s^-$). Lemma \ref{nblem} and Example \ref{minresex} serve as preparations for this step.

\begin{lemma}\label{nblem}
Under the assumptions of Lemma \ref{mgeom}, denote by $M^{\circ} = M - \{m_s\}_{s\in S_2}$ the smooth locus of $M$. Write $\Sigma=\Sigma_s^*$ for $*=\pm$ and $\Sigma^{\circ}=\Sigma-\{m_s\}$. 

(1) The normal bundle $N_{\Sigma^{\circ}/M^{\circ}}$ is isomorphic to $\O_{\Sigma^{\circ}}(-2)$.

(2) $\O_{\Sigma}(-K_M) \cong \O_{\Sigma}(1)$.
\end{lemma}
\begin{proof}
(1) There is an exact sequence of normal bundles
\begin{equation*}
    0\to N_{\Sigma^{\circ}/M^{\circ}} \to N_{\Sigma^{\circ}/\Gr_S(2, \E)} \to (N_{M^{\circ}/\Gr_S(2,\E)})|_{\Sigma^{\circ}} \to 0.
\end{equation*}
From inclusions $\Sigma^{\circ}\subset \Gr(2,4) \subset\Gr_S(2, \E)$, a similar exact sequence of normal bundles give
\begin{equation*}
    N_{\Sigma^{\circ}/\Gr_S(2, \E)} \cong N_{\Sigma^{\circ}/\Gr(2,4)} \oplus (T_s S \otimes \O_{\Sigma^{\circ}}) \cong \R^\vee|_{\Sigma^{\circ}} \oplus \O_{\Sigma^{\circ}}^2,
\end{equation*}
where $T_s S$ is the tangent space of $S$ at $s$. By applying (\ref{msec}), we get
\begin{equation*}
    N_{M^{\circ}/\Gr_S(2,\E)} \cong \Sym^2(\R^\vee)\otimes \pi_{\Gr}^*\L,
\end{equation*}
where $\pi_{\Gr}\colon \Gr_S(2, \E)\to S$ is the projection map. They imply
\begin{equation*}
    \det(N_{\Sigma^{\circ}/\Gr_S(2, \E)}) \cong \O_{\Sigma^{\circ}}(1), \quad \det((N_{M^{\circ}/\Gr_S(2,\E)})|_{\Sigma^{\circ}}) \cong \O_{\Sigma^{\circ}}(3).
\end{equation*}

Hence, $N_{\Sigma^{\circ}/M^{\circ}} \cong \det( N_{\Sigma^{\circ}/M^{\circ}}) \cong \O_{\Sigma^{\circ}}(-2)$.

(2) From (1), we have $\O_{\Sigma^{\circ}}(\Sigma) \cong \O_{\Sigma^{\circ}}(-2)$. We get 
\begin{align*}
    \O_{\Sigma^{\circ}}(-K_M) & \cong \O_{\Sigma^{\circ}}(-K_M-\Sigma) \otimes \O_{\Sigma^{\circ}}(\Sigma)\\
    & \cong \O_{\Sigma^{\circ}}(-K_{\Sigma}) \otimes \O_{\Sigma^{\circ}}(-2)\\
    & \cong \O_{\Sigma^{\circ}} (1).
\end{align*}
Since $M$ is Gorenstein, we have that $\O_{\Sigma}(-K_M)$ is a line bundle and the isomorphism extends to $\Sigma$.
\end{proof}

\begin{ex}\label{minresex}
Let $X =\{xy-zw=0\} \subset \P^4$ be the nodal quadric $3$-fold. It is a cone over $\P^1 \times \P^1$ with the vertex $O=\{x=y=z=w=0\}$. Let $\varphi\colon X \dashrightarrow \P^1 \times \P^1$ be the projection from the vertex. Write $\Sigma_t=\varphi^{-1}(\P^1\times \{t\}) \cong \P^2, t\in \P^1$. Let $Y$ be the minimal resolution of $X$. Then $Y\cong \P(\O_{\P^1}\oplus \O_{\P^1}(-1)^2)\subset \P^4 \times \P^1$ and $Y \cong \Bl_{\Sigma_t} X$ is the blow up of $X$ along $\Sigma_t$ for every $t\in \P^1$. Fix a point $0\in \P^1$ and let $\psi\colon X\dashrightarrow \P^1$ be the linear projection from $\Sigma_0$. The resolution of the indeterminacy of $\psi$ gives
\begin{equation*}
    \begin{tikzcd}[column sep = small]
    &  Y\arrow{ld}[swap]{f} \arrow{rd}{g} &\\
    X \arrow[dotted]{rr}{\psi} & & \P^1,
    \end{tikzcd}
\end{equation*}
where $f$ is the blow-up of $X$ along $\Sigma_0$. For every $t\in \P^1$, the pre-image $\ts_t=f^{-1}(\Sigma_t)$ is the Hirzebruch surface $\Bl_O(\Sigma_t)$, where $O\in \Sigma_t$ is the vertex of $X$. Let $H$ and $h$ be the pull-backs of hyperplane classes of $X$ and $\P^1$ to $Y$, respectively. Then we have the relation $\ts_t=H-h$. Let $l \cong \P^1$ be the exceptional locus of $f$. Then $\O_l(\ts_t) \cong \O_l(H-h) \cong \O_l(-1)$.
\end{ex}

The locus $S_2$ can be embedded into the double cover $\wt{S}$ and for $s\in S_2\subset \wt{S}$, we have that $M_s = \Sigma_s^+\cup \Sigma_s^-$ is also the scheme-theoretic fiber of $\tau\colon M\to \wt{S}$. Moreover, $\wt{S}$ has ordinary double points at $S_2$.

Recall $m_s=\Sigma_s^+\cap \Sigma_s^-$ and $M^{\circ} = M - \{m_s\}_{s\in S_2}$. By Lemma \ref{mgeom}, $M$ has a finite number of ordinary double points $\{m_s\}_{s\in S_2}$. That is, near the point $m_s$, the $3$-fold $M$ is the nodal quadric. Since $\Sigma_s^+$ and $\Sigma_s^-$ intersect at only one point, they are the planes lying over the same rulings of $\P^1\times \P^1$. Let 
\begin{equation*}
    \xi\colon \wt{M} \to M
\end{equation*}
be the blow-up of $M$ along $\bigsqcup_{s\in S_2} \Sigma_s^+$. Then $\wt{M}$ is a small resolution of $M$ and $\xi^{-1}(M^{\circ}) \to M^{\circ}$ is an isomorphism. 

Write $l_s =\xi^{-1}(m_s) \cong \P^1$ and let $*=\pm$. Example \ref{minresex} tells us that 
\begin{equation*}
    \ts_s^* := \xi^{-1}(\Sigma_s^*) \cong \Bl_{m_s} \Sigma_s^*
\end{equation*}
and $l_s$ is the $(-1)$-curve on $\Bl_{m_s} \Sigma_s^*$. Denote the fiber classes of the projection $\ts_s^* \to l_s$ by $h_s^*$. Then 
\begin{equation*}
    h_s^*\cdot l_s=1, \quad l_s^2=-1, \quad (h_s^*)^2=0
\end{equation*}
on $\ts_s^*$. In addition, $\xi^*\O_{\Sigma_s^*}(1) \cong \O_{\ts_s^*}(h_s^* +l_s)$.

\begin{prop}\label{birtran}
Assume $\kk$ is algebraically closed and $\cchar(\kk)=0$. 

(1) There exists a relative contraction map $\xi_+\colon \wt{M} \to M^+$ over $\wt{S}$ where $\xi_+$ is an isomorphism on $\wt{M}-\bigsqcup_{s\in S_2} \ts_s^+$ and $\xi_+|_{\ts_s^+}$ is the projection onto $l_s$. Namely, we have birational morphisms
\begin{equation*}
    M\xleftarrow{\xi} \wt{M} \xrightarrow{\xi_+} M^+
\end{equation*}
over $\wt{S}$ and the corresponding fibers over $s\in S_2\subset \wt{S}$ are
\begin{equation*}
    \Sigma_s^+\cup \Sigma_s^-\leftarrow \ts_s^+ \cup \ts_s^- \to \ts_s^-.
\end{equation*}
Furthermore, $M^+$ is smooth and $\xi_+$ is the blow-up of $M^+$ along $\bigsqcup_{s\in S_2} l_s$.

(2) Let $\eta\colon S^+=\Bl_{S_2} \wt{S} \to \wt{S}$ be the resolution of $\wt{S}$. Then the map $\tau'\colon M^+\to \wt{S}$ obtained from (1) fits into the commutative diagram
\begin{equation*}
    \begin{tikzcd}[column sep = small]
    & \wt{M}\arrow{ld}[swap]{\xi} \arrow{rd}{\xi_+} & \\
    M \arrow{d}[swap]{\tau} & & M^+ \arrow{d}{\tau_+} \arrow{lld}[swap]{\tau'}\\
    \wt{S} & & S^+ \arrow{ll}{\eta} 
    \end{tikzcd}
\end{equation*}
and $\tau_+$ is a smooth conic bundle.
\end{prop}
\begin{proof}
Denote by $\wt{S}^{\circ}=\wt{S}\backslash S_2$ and let $*=\pm$.

(1) By Lemma \ref{nblem} (1), $N_{\ts_s^*/\wt{M}} \cong \O(-2h_s^*+nl_s)$ for some $n\in\ZZ$. Since $\O_l(\ts_s^*) \cong \O_{l_s}(-1)$, we have $n=-1$ and 
\begin{equation}\label{nbeq}
    N_{\ts_s^*/\wt{M}} \cong \O_{\ts_s^*}(\ts_s^*) \cong \O_{\ts_s^*}(-2h_s^*-l_s).
\end{equation}
Let 
\begin{equation}\label{condiv}
    D=-K_{\wt{M}}-\sum_{s\in S_2}\ts_s^- = -\xi^*K_{M}-\sum_{s\in S_2}\ts_s^-.
\end{equation}
Denote $\wt{\tau}=\tau\circ\xi\colon \wt{M} \to \wt{S}$. We claim that $D$ is $\wt{\tau}$-nef and $D-K_{\wt{M}}$ is $\wt{\tau}$-ample. See \cite[Tag 01VH]{stacks-project} for the definition of relative ampleness and recall that a divisor is relative nef if its intersection with every curve in the fiber is non-negative.

Note that $\wt{\tau}^{-1}(\wt{S}^{\circ})\cong \tau^{-1}(\wt{S}^{\circ}) \to \wt{S}^{\circ}$ is a smooth conic bundle and $D|_{\wt{\tau}^{-1}(\wt{S}^{\circ})} =-K_{\wt{\tau}^{-1}(\wt{S}^{\circ})}$ is relative ample. It suffices to study $D$ and $D-K_{\wt{M}}$ on $\ts_s^+\cup \ts_s^-$. The relative nefness and ampleness can be checked on curve classes $h_s^{\pm}$ and $l_s$. By Lemma \ref{nblem} (2),
\begin{equation*}
    \O_{\ts_s^*}(-K_{\wt{M}}) \cong \xi^* \O_{\Sigma_s^*}(-K_M) \cong \O_{\ts_s^*}(h_s^*+l_s).
\end{equation*}
Hence, $(-K_{\wt{M}})\cdot h_s^* = 1$, $(-K_{\wt{M}})\cdot l_s=0$ and
\begin{equation}\label{nefness}
    D\cdot h_s^+=0, \quad D\cdot h_s^-=2, \quad D\cdot l_s=1.
\end{equation}
From these, we get $D$ is relative nef, and by Kleiman's ampleness criterion (Theorem 1.44 in \cite{kmbirgeo}), $D-K_{\wt{M}}$ is relative ample. By relative basepoint-free theorem (Theorem 3.24 in \cite{kmbirgeo} and this is where $\cchar(\kk)$=0 is used), $mD$ is $\wt{\tau}$-free for $m\gg 0$. We can construct $\xi_+\colon \wt{M} \to M^+$ by taking the Stein factorization of $|mD|, m\gg 0$. From Equations (\ref{nefness}), we get that $\xi_+$ has the required properties. From (\ref{nbeq}), we get $N_{\ts_s^+/\wt{M}}|_{h_s^+} \cong \O_{h_s^+}(-1)$. Then \cite{funainv} or \cite[Theorem 2.3]{andoexray} implies that $M^+$ is smooth and $\xi_+$ is the blow-up.

(2) Note that from (1), we have $M^+$ is a smooth $3$-fold. The fiber 
of $\tau'\colon M^+\to \wt{S}$ over $s\in S_2$ is a Cartier divisor $\ts_s^-$ of $M^+$. From the universal property of blowing up, we get that $\tau'$ factors as $M^+\xrightarrow{\tau_+} S^+\xrightarrow{\eta} \wt{S}$. By construction, $\eta^{-1}(\wt{S}^{\circ})\cong \wt{S}^{\circ}$, and $\tau_+^{-1}(\wt{S}^{\circ}) \to \wt{S}^{\circ}$ is isomorphic to $\tau^{-1}(\wt{S}^{\circ}) \to \wt{S}^{\circ}$. We claim that $\tau_+|_{\ts_s^-}\colon \ts_s^-\to \eta^{-1}(s)$ is the projection from the Hirzebruch surface onto $\P^1$.

Note that $\xi_+^{-1}(\ts_s^-)=\ts_s^+ \cup \ts_s^-$ and $\xi_+|_{\ts_s^-}$ is an isomorphism. Then from (\ref{nbeq}), we get
\begin{equation}\label{nbm+}
    N_{\ts_s^-/M^+} \cong \O_{M^+}(\ts_s^-)|_{\ts_s^-}\cong \O_{\wt{M}}(\ts_s^+ + \ts_s^-)|_{\ts_s^-} \cong \O_{\ts_s^-}(-2h_s^-).
\end{equation}
By construction of $\tau_+$, the pull-back of $\O_{\eta^{-1}(s)}(1)\cong \O_{\P^1}(2)$ is the conormal bundle $N_{\ts_s^-/M^+}^\vee$. Thus, the pull-back of $\O_{\P^1}(1)$ is $\O_{\ts_s^-}(h_s^-)$ and $\tau_+|_{\ts_s^-}$ is the said projection.

Lastly, we show $\tau_+$ is a smooth conic bundle. Recall that a map $f\colon X\to Y$ is a smooth conic bundle if each geometric fiber $X_y$ is $\P^1$ and there exists a line bundle $L$ on $X$ such that $L|_{X_y}\cong \O_{\P^1}(2)$. We have seen the geometric fibers of $\tau_+$ are $\P^1$'s. Now we will show that there is a line bundle $L$ on $M^+$ such that $\xi_+^*L \cong \O_{\wt{M}}(D)$ with $D$ defined in (\ref{condiv}) and $L$ makes $\tau_+$ a smooth conic bundle.

Since $\xi_+$ is a smooth blow-up from (1), the SOD of $\D(\wt{M})$ obtained from the blow-up formula implies that $\O_{\wt{M}}(D)$ is the pull-back of a line bundle $L$ on $M^+$ if $\O_{\wt{M}}(D)|_{\ts_s^+}$ is the pull-back of a line bundle on $l_s$. Computations in (1) give
\begin{equation*}
    \O_{\ts_s^+}(D) \cong  \O_{\ts_s^+}(h_s^+) \cong \xi_+^* \O_{l_s}(1), \quad \O_{\ts_s^-}(D) \cong \O_{\ts_s^-}(3h_s^-+2l_s).
\end{equation*}
Therefore, such $L$ exists and write $L=\O_{M^+}(D^+)$. Note that $\O_M(-K_M)$ restricted to $\tau^{-1}(\wt{S}^{\circ})$ makes $\tau^{-1}(\wt{S}^{\circ}) \to \wt{S}^{\circ}$ a smooth conic bundle. Since $\O_{M^+}(D^+)$ and $\O_M(-K_M)$ restricted to $\tau_+^{-1}(\wt{S}^{\circ}) \cong \tau^{-1}(\wt{S}^{\circ})$ are isomorphic and
\begin{equation*}
    D^+\cdot h_s^- =D \cdot h_s^-= 2,
\end{equation*}
we have that $\O_{M^+}(D^+)$ makes $\tau_+$ a smooth conic bundle.
\end{proof}

For the rest of the section, we will show that $\A_{\Q}\cong \D(S^+, \A^+)$ where $\A^+$ is Brauer equivalent to the Azumaya algebra corresponding to $\tau_+\colon M^+\to S^+$. Recall that $\R$ is the universal subbundle on $\Gr_S(2,\E)$. Let $\R_M$ be the restriction of $\R$ on $M$. Recall the associated left Clifford ideals $\I_n^{\R_M}, n\in\ZZ$ in Definition \ref{clideal}.

\begin{lemma}\label{ires}
For every $n\in \ZZ$ and $*=\pm$, we have  $(\xi^*\I_n^{\R_M})|_{\ts_s^*}\cong \O_{\ts_s^*} \oplus \O_{\ts_s^*}(-h_s^*-l_s)$.
\end{lemma}
\begin{proof}
As with the filtration (\ref{i0fil}), for $m\in\ZZ$, we have the short exact sequence
\begin{equation*}
    0\to \det(\R_M) \otimes \rho^*\L^{m-1} \to \I_{2m}^{\R_M} \to \rho^*(\det(\E)\otimes \L^{m-2}) \to 0
\end{equation*}
and $\I_{2m+1}^{\R_M}\cong \det(\R_M)\otimes (\rho^*\E/\R_M) \otimes \rho^*\L^{m-1}$. Note that $\Sigma_s^*\cong \Gr(2,3)\subset \Gr(2,4)$ and the universal quotient on $\Gr(2,4)$ restricted to $\Gr(2,3)$ is $\O\oplus \O(1)$. Then for every $n\in \ZZ$ we get $\I_n^{\R_M}|_{\Sigma_s^*}\cong \O \oplus \O(-1)$ and the result follows.
\end{proof}

\begin{lemma}\label{jlem}
For every $n\in \ZZ$, there exists a rank $2$ vector bundle $\J_n$ on $M^+$ that fits into short exact sequences
\begin{equation}\label{ijseq}
    0\to \xi_+^*\J_n\to \xi^*\I_n^{\R_M} \to \bigoplus_{s\in S_2} \O_{\ts_s^+}(-h_s^+ -l_s) \to 0,
\end{equation}
\begin{equation}\label{ijdual}
    0\to \xi^*(\I_n^{\R_M})^\vee \to \xi_+^*(\J_n)^\vee\to \bigoplus_{s\in S_2} \O_{\ts_s^+}(-h_s^+) \to 0.
\end{equation}
\end{lemma}
\begin{proof}
By Lemma \ref{ires}, we have $(\xi^*\I_n^{\R_M})|_{\ts_s^+}\cong \O \oplus \O(-h_s^+-l_s)$. We can construct a surjection
\begin{equation*}
     \xi^*\I_n^{\R_M} \twoheadrightarrow (\xi^*\I_0^{\R_M})|_{\ts_s^+} \twoheadrightarrow \bigoplus_{s\in S_2} \O_{\ts_s^+}(-h_s^+ -l_s)
\end{equation*}
and denote its kernel by $\F$. Restricting this short exact sequence to $\ts_s^+$, we have
\begin{equation*}
    0 \to \O_{\ts_s^+}(-h_s^+ -l_s) \otimes N_{\ts_s^+/\wt{M}}^\vee \to \F|_{\ts_s^+} \to \O_{\ts_s^+}\oplus \O_{\ts_s^+}(-h_s^+ -l_s)\to \O_{\ts_s^+}(-h_s^+ -l_s) \to 0.
\end{equation*}
Since $N_{\ts_s^+/\wt{M}}^\vee \cong 
O_{\ts_s^+}(2h_s^+ +l_s)$ by (\ref{nbeq}), we have
\begin{equation}\label{jres}
    \F|_{\ts_s^+} \cong \O_{\ts_s^+} \oplus \O_{\ts_s^+}(h_s^+)    
\end{equation}
and this is the pull-back of $\O_{l_s}\oplus \O_{l_s}(1)$. Since $\xi_+\colon \wt{M} \to M^+$ is a smooth blow-up along $\bigsqcup l_s$ with exceptional locus $\bigsqcup \ts_s^+$ by Proposition \ref{birtran} (1), we deduce $\F\cong \xi_+^*\J_n$ for some rank $2$ bundle $\J_n$ on $M^+$. This gives (\ref{ijseq}). Since 
\begin{equation*}
    R\sHom(\O_{\ts_s^+}(-h_s^+ -l_s), \O_{\wt{M}}) \cong R\sHom(\O_{\ts_s^+}(-h_s^+ -l_s), N_{\ts_s^+/\wt{M}}[-1])\cong \O_{\ts_s^+}(-h_s)[-1],
\end{equation*}
we get (\ref{ijdual}) by taking the dual of (\ref{ijseq}).
\end{proof}

From Proposition \ref{birtran} (2), we get commutative diagrams
\begin{equation}\label{commdiag}
    \begin{tikzcd}[sep = small]
    & \wt{M}\arrow{ld}[swap]{\xi} \arrow{rd}{\xi_+} & \\
    M \arrow{dd}[swap]{\tau} \arrow{rd}{\rho} & & M^+ \arrow{dd}{\tau_+} \arrow{ld}[swap]{\rho_+}\\
    & S & \\
    \wt{S} \arrow{ru}[swap]{\alpha} & & S^+ \arrow{ll}{\eta} \arrow{lu}{\alpha_+}
    \end{tikzcd}
\end{equation}

\begin{lemma}\label{ijpushf}
For every $n\in \ZZ$,

(1) $R\rho_*\I_n^{\R_M}=0$ and $R\rho_*\sEnd(\I_n^{\R_M})\cong \B_0$ as sheaves of algebras;

(2) $R(\rho_+)_*\J_n=0$ and $R(\rho_+)_*\sEnd(\J_n)\cong \B_0$ as sheaves of algebras.
\end{lemma}
\begin{proof}
(1) In \cite[\S 3]{kuzline}, the right $\B_0$-module $\fS_n$ is constructed as the cokernel of $\R_M\otimes \rho^*\B_{n-1}\to \rho^*\B_n$. Comparing it with Lemma \ref{coklem}, we have $\fS_n\cong \I_n^{\circ \R_M}\otimes\det(\R_M^\vee)\otimes \rho^*\L$. By Lemma \ref{clidealdual}, we have
\begin{equation*}
    \I_n^{\R_M}\cong \fS_{-n}^\vee \otimes \rho^*(\det(\E)\otimes (\L^\vee)^2).
\end{equation*}
Then Corollary 3.5, 3.6 in \textit{loc. cit.} give (1).

(2) Write $\I_n=\I_n^{\R_M}$. Observe from the diagram (\ref{commdiag}) that $\rho_+\xi_+=\rho\xi$ restricted to $\ts_s^+$ is the map $\ts_s^+ \to s=\Spec(\kk)$. We first recall some results on the blow-up $X$ of $\P^2$ at a point that are needed for the proof. The blow-up $X\subset \P^2\times \P^1$ is a divisor $\O_{\P^2\times \P^1}(1,1)$. Denote by $\O_X(a,b), a,b \in\ZZ$ the restriction of $\O_{\P^2\times \P^1}(a,b)$ to $X$. Then we easily see from the short exact sequence
\begin{equation*}
    0\to \O_{\P^2\times \P^1}(a-1,b-1)\to \O_{\P^2\times \P^1}(a,b)\to \O_X(a,b) \to 0 
\end{equation*}
that $H^\bullet(X, \O_X(a,b))=0$ for $a=-1$ or $(a,b)=(0,-1)$. Observe also that
\begin{equation*}
    R(\rho_+)_* = R(\rho_+)_* R(\xi_+)_* L\xi_+^* = R\rho_* R\xi_* L\xi_+^*.
\end{equation*}
Applying $R\rho_* R\xi_*$ to the sequence (\ref{ijseq}), we have the last term vanishes and thus $R(\rho_+)_*\J_n\cong R\rho_*\I_n=0$.

We have seen from Lemma \ref{ires} and (\ref{jres}) that
\begin{equation*}
    (\xi_+^*\J_n)|_{\ts_s^+} \cong \O_{\ts_s^+} \oplus \O_{\ts_s^+}(h_s^+), \quad (\xi^*\I_n)|_{\ts_s^+}\cong \O_{\ts_s^+} \oplus \O_{\ts_s^+}(-h_s^+-l_s).
\end{equation*}
Tensoring (\ref{ijseq}) with $\xi_+^*(\J_n)^\vee$ and (\ref{ijdual}) with $\xi^*\I_n$, respectively, we get
\begin{equation*}
    \begin{split}
        & 0\to \xi_+^*\sEnd(\J_n) \to \xi^*\I_n\otimes \xi_+^*(\J_n)^\vee \to \bigoplus_{s\in S_2} (\O_{\ts_s^+}(-h_s^+ -l_s) \oplus \O_{\ts_s^+}(-2h_s^+ -l_s)) \to 0,\\
        & 0\to \xi^*\sEnd(\I_n) \to \xi^*\I_n\otimes \xi_+^*(\J_n)^\vee \to \bigoplus_{s\in S_2} (\O_{\ts_s^+}(-h_s^+) \oplus \O_{\ts_s^+}(-2h_s^+ -l_s)) \to 0.
    \end{split}
\end{equation*}
Applying $R(\rho_+)_* R(\xi_+)_* = R\rho_* R\xi_*$ to the sequence above, we have last terms in both sequences vanish. Then
\begin{equation*}
    R(\rho_+)_*\sEnd(\J_n) \cong R\rho_* R\xi_*(\xi^*\I_n\otimes \xi_+^*(\J_n)^\vee) \cong R\rho_*\sEnd(\I_n)\cong \B_0.
\end{equation*}
\end{proof}

\begin{prop}\label{azprop}
(1) For every $n\in \ZZ$, we have $\J_n|_{\ts_s^-} \cong \O_{\ts_s^-}(-h_s^- -l_s) \oplus \O_{\ts_s^-}(-l_s)$.

(2) Let $\B^+$ be the Azumaya algebra on $S^+$ that corresponds to the smooth conic bundle $\tau_+\colon M^+\to S^+$. Then $\sEnd(\J_0)\cong \tau_+^*\A^+$ for some Azumaya algebra $\A^+$ on $S^+$ that is Brauer equivalent to $\B^+$.

(3) $R(\alpha_+)_*\A^+\cong \B_0$ as sheaves of algebras.
\end{prop}
\begin{proof}
(1) By restricting the sequence (\ref{ijseq}) to $\ts_s^-$ and using Lemma \ref{ires}, we obtain
\begin{equation*}
    0\to \J_n|_{\ts_s^-} \to \O_{\ts_s^-}\oplus \O_{\ts_s^-}(-h_s^- -l_s)\to \O_{l_s}\to 0.
\end{equation*}
Then $\J_n|_{\ts_s^-}\cong \O_{\ts_s^-}(-h_s^- -l_s) \oplus \O_{\ts_s^-}(-l_s)$ or $\O_{\ts_s^-}(-h_s^- -2l_s) \oplus \O_{\ts_s^-}$. We will show the latter is impossible.

Let $s\in S_2\subset S$. By Lemma \ref{ijpushf} (2),
\begin{equation*}
    0=\Ext^\bullet(\O_s, R(\rho_+)_*\J_n) \cong \Ext^\bullet(L\rho_+^*\O_s, \J_n).
\end{equation*}
Denote by $\cH^i= \cH^i(L\rho_+^*\O_s)$. Then $\cH^i=0$ for $i>0$ and $\cH^0\cong \O_{\ts_s^-}$. Consider the spectral sequence
\begin{equation}\label{sseq}
    \Ext^j(\cH^i,\J_n)\Rightarrow \Ext^{j-i}(L\rho_+^*\O_s, \J_n)=0.
\end{equation}
By Serre duality on $M^+$, we get
\begin{equation}\label{extiso}
    \Ext^j(\cH^i,\J_n) \cong \Ext^{3-j}(\J_n, \cH^i\otimes \omega_{M^+})^\vee \cong H^{3-j}(M^+, \J_n^\vee \otimes\cH^i\otimes \omega_{M^+})^\vee,
\end{equation}
where $\omega_{M^+}$ is the canonical line bundle. By Equation (\ref{nbm+}), we get
\begin{equation*}
    \omega_{M^+}|_{\ts_s^-} \cong \omega_{\ts_s^-} \otimes N_{\ts_s^-/M^+}^\vee \cong \O_{\ts_s^-}(-h_s^- -2l_s).
\end{equation*}
Since $\cH^i$ is supported on $\ts_s^-$, the right hand side in (\ref{extiso}) is $0$ if $j\notin \{1,2,3\}$. Thus, the line $j=2$ is stable in the spectral sequence (\ref{sseq}). If $\J_n|_{\ts_s^-}\cong \O_{\ts_s^-}(-h_s^- -2l_s) \oplus \O_{\ts_s^-}$, then
\begin{align*}
    \Ext^2(\cH^0, \J_n)& \cong H^1(M^+, (\J_n^\vee)|_{\ts_s^-} \otimes \omega_{M^+})^\vee\\
    & \cong H^1(M^+, \O_{\ts_s^-} \oplus \O_{\ts_s^-}(-h_s^- -2l_s))^\vee\cong \kk.
\end{align*}
This implies $\Ext^2(L\rho_+^*\O_s, \J_n)\neq 0$, which is a contradiction. 

(2) Notice that if there are rank $2$ vector bundles $\F_i, i=1,2$ on $M^+$ such that $\F_i$ restricted to each geometric fiber $M_t^+, t\in S^+$ of $\tau_+$ is $\O_{\P^1}(-1)^2$, then $\sEnd(\F_i)$ and $\sHom(\F_1, \F_2)$ restricted to $M_t^+$ are trivial. Thus, there are Azumaya algebras $\A_i, i=1,2$ and a vector bundle $\cV$ on $S^+$ such that $\sEnd(\F_i)\cong \tau_+^*\A_i$ and $\sHom(\F_1, \F_2) \cong \tau_+^*\cV$. Since
\begin{equation*}
    \sEnd(\F_1^\vee) \otimes \sEnd(\F_2)\cong \sEnd(\sHom(\F_1, \F_2)),
\end{equation*}
we have
\begin{equation*}
    \A_1^{\op} \otimes\A_2\cong \sEnd(\cV)
\end{equation*}
where $\A_1^{\op}$ is the opposite algebra of $\A_1$. That is, $\A_1, \A_2$ are Brauer equivalent. 

By construction, $\tau_+^*\B^+$ is trivial, i.e., the endomorphism of some rank $2$ vector bundle $\F$ on $M^+$. Since $\O_{M^+/S^+}(1)|_{M_t^+}\cong \O_{\P^1}(2)$, we can choose $\F$ such that $\F|_{M_t^+}\cong \O_{\P^1}(-1)^2$. We claim that $\J_n$ is such rank $2$ vector bundle for every $n\in \ZZ$.

By Lemma \ref{ijpushf} (2), $R(\rho_+)_*\J_n=0$. Since $\alpha_+\colon S^+\to S$ restricted to $S^+\backslash \bigsqcup l_s$ is finite, we have $R(\tau_+)_*\J_n=0$ on $S^+\backslash \bigsqcup l_s$ and thus $\J_n|_{M_t^+}\cong \O_{\P^1}(-1)^2$ for every $t\in S^+\backslash \bigsqcup l_s$. On the other hand, from (1), we get that $\J_n$ restricted to the fibers of $\ts_s^- \to l_s$ is also $\O_{\P^1}(-1)^2$. Thus, we get $\sEnd(\J_0)\cong \tau_+^*\A^+$ for some Azumaya algebra $\A^+$ on $S^+$ that is Brauer equivalent to $\B^+$.

(3) Observe from the diagram (\ref{commdiag}) that
\begin{equation*}
    R(\alpha_+)_*\cong R(\alpha_+)_* R(\tau_+)_* L \tau_+^* \cong R(\rho_+)_* L \tau_+^*.
\end{equation*}
Then by (2) and Lemma \ref{ijpushf} (2), we get
\begin{equation*}
    R(\alpha_+)_* \A^+ \cong R(\rho_+)_* \tau_+^* \A^+ \cong R(\rho_+)_* \sEnd(\J_0) \cong \B_0.
\end{equation*}
\end{proof}

\begin{theorem}\label{main3}
Assume $\kk$ is algebraically closed and $\cchar(\kk)=0$. Let $p\colon \Q\to S$ be a flat quadric surface bundle where $\Q$ is smooth and $S$ is a smooth surface over $\kk$. Then there is a semiorthogonal decomposition
\begin{equation}\label{sod3}
    \D(\Q) =\langle \D(S^+, \A^+), p^*\D(S), p^*\D(S)\otimes \O_{\Q/S}(1) \rangle 
\end{equation}
where $S^+$ is the resolution of the double cover $\wt{S}$ over $S$ ramified along the (first) degeneration locus $S_1$ and $\A^+$ is an Azumaya algebra on $S^+$. In addition, the Brauer class $[\A^+]\in \Br(S^+)$ is trivial if and only if $p\colon \Q \to S$ has a rational section.
\end{theorem}
\begin{proof}
Recall that by the SOD (\ref{kuzsod}), the non-trivial component of $\D(\Q)$ is equivalent to $\D(S,\B_0)$. To get the SOD (\ref{sod3}), it suffices to show that $R(\alpha_+)_*\colon \D(S^+, \A^+)\to \D(S,\B_0)$ is an equivalence. The proof of this is similar to that of Proposition \ref{nceqprop}. By Proposition \ref{azprop} (3) and the projection formula (Proposition \ref{projf}), we get $R(\alpha_+)_*L\alpha_+^* \cong \id$. On the other hand, since $S^+\backslash \bigsqcup l_s\cong \wt{S}\backslash S_2 \to S\backslash S_2$ is finite, we have that $R(\alpha_+)_*\F =0$ for $\F\in \Coh(S^+, \A^+)$ implies that $\F$ is supported on $\bigsqcup l_s$. By Proposition \ref{azprop} (1)(2),
\begin{equation*}
    \tau_+^*(\A^+|_{l_s}) \cong \sEnd(\J_0)|_{\ts_s^-} \cong \sEnd(\O_{\ts_s^-} \oplus \O_{\ts_s^-}(-h_s^-)) \cong \tau_+^*(\sEnd(\O_{l_s}\oplus \O_{l_s}(-1))). 
\end{equation*}
Hence, $\A^+|_{l_s} \cong \sEnd(\O_{l_s}\oplus \O_{l_s}(-1))$. Now the proof in Proposition \ref{nceqprop} shows that $\F=0$ and $L\alpha_+^* R(\alpha_+)_* \cong \id$. 

Let $U=S\backslash S_2$. Since $S^+$ is smooth and integral, we have that the composition $\Br(S^+) \to \Br(\alpha_+^{-1}(U))\to \Br(\kk(S^+))$ is injective. Thus, the restriction $\Br(S^+) \to \Br(\alpha_+^{-1}(U))$ is injective. Observe from the diagram (\ref{commdiag}) that $\rho^{-1}(U)\to \alpha^{-1}(U)$ is isomorphic to $\rho_+^{-1}(U) \to \alpha_+^{-1}(U)$. By Proposition 2.15 in \cite{ksgroring}, $[\A^+]=0$ if and only if $p$ has a rational smooth ({\it non-degenerate} in {\it loc. cit.}) section. Lastly, since $\Q$ and $S$ are smooth, every section of $p$ is smooth by Lemma 1.3.2 in \cite{abbqfib}.
\end{proof}
\section{Examples}\label{exsec}
In this section, we apply the main theorems to examples. We start with some remarks on nodal quintic del Pezzo $3$-folds and cubic $4$-folds containing a plane, and then we consider the most important cases of complete intersections of quadrics.

\begin{ex}[\cite{xienodaldp5}] \label{nodaldp5}
Let $X_m\subset \P^6, m=1,2,3$ be the nodal quintic del Pezzo $3$-folds with $m$ nodes. Let $x\in X_m$ be a node. Then the embedded projective tangent space $T_x X_m$ is isomorphic to $\P^4$. The linear projection $X_m\dashrightarrow \P^1$ from $T_x X_m$ induces the map $p_m\colon Y_m \to \P^1$ where $f_m\colon Y_m=\Bl_{\P^4\cap X_m} X_m \to X_m$ is the blow-up. In fact, $f_m$ is a (partial) resolution of $X_m$ at the nodal point $x$. We have that $p_m$ is a quadric surface bundle where $p_1, p_2$ have simple degeneration and $p_3$ has a fiber of corank $2$. In addition, the exceptional locus of $f_m$ is a smooth section of $p_m$. The hyperbolic reduction $C_m$ with respect to the smooth section is a nodal chain of $m$ $\P^1$'s. Hence, the residual category of $Y_m$ is equivalent to $\D(C_m)$ by Theorem \ref{main1} or \ref{main2}.
\end{ex}

\begin{ex}[\cite{moc8}] \label{cubic4}
Let $X\subset \P^5$ be a smooth cubic $4$-fold containing a plane and let $Y=\Bl_{\P^2} X$ be the blow-up of $X$ along the plane. The linear projection $X\dashrightarrow \P^2$ from the plane induces $Y\to \P^2$ where $Y\to \P^2$ is a quadric surface bundle possibly with fibers of corank $2$. The Kuznetsov component of $X$ is equivalent to the residual category of $Y$. By Theorem \ref{main3}, the residual category of $Y$ is equivalent to the twisted derived category of a smooth K3 surface. This K3 surface is obtained as the resolution of the double cover over $\P^2$ ramified along a nodal sextic curve.

The same result proved in \cite{moc8} uses the result of quadric surface bundles over smooth $3$-folds in \cite{kuzline}. In order to use \cite{kuzline}, $X$ is described as a hyperplane section of a smooth cubic $5$-fold containing the plane such that the induced quadric surface bundle over $\P^3$ satisfies the required hypotheses on degeneration loci. Section \ref{surfbsec} in this paper provides a more direct proof.
\end{ex}

Now we consider applications to complete intersections of quadrics. Let 
\begin{equation*}
    X^{n,k}= \bigcap_{i=1}^k Q_i \subset \P^{n+1}, \quad k \leqslant n 
\end{equation*}
be the complete intersection of $k$ quadrics $Q_i=\{q_i=0\}\subset \P^{n+1}$. Let 
\begin{equation*}
    p^{n,k}\colon \Q^{n,k}\to \P^{k-1}
\end{equation*}
be the corresponding net of quadrics, i.e., the fiber over $[a_1:\dots:a_k]\in \P^{k-1}$ is $\{\sum_{i=1}^k a_iq_i=0\}$. Then $\dim(X^{n,k})=n+1-k$ and $p^{n,k}$ is a flat quadric bundle of relative dimension $n$ whose associated quadratic form is 
\begin{equation*}
    q^{n,k}\colon \O_{\P^{k-1}}^{n+2}\to \O_{\P^{k-1}}(1).
\end{equation*}
When $X^{n,k}$ is Fano or Calabi-Yau, i.e., $n\geqslant 2k-2$, Theorem 5.5 in \cite{kuzqfib} states that there is a semiorthogonal decomposition
\begin{equation}\label{cisod}
    \D(X^{n,k})=\langle \D(\P^{k-1}, \B_0^{n,k}), \O(1), \O(2), \dots, \O(n+2-2k) \rangle
\end{equation}
where $\B_0^{n,k}$ is the even Clifford algebra of $p^{n,k}\colon \Q^{n,k}\to \P^{k-1}$.

\begin{prop}
Assume that $n\geqslant 2k-2$, $n$ is even and write $n=2m+2$. Assume that the smooth locus of $X^{n,k}$ contains a $\P^m$. Then $\P^m\times \P^{k-1}\subset \Q^{n,k}$ is a smooth $m$-section of $p^{n,k}\colon \Q^{n,k}\to \P^{k-1}$ as in Definition \ref{regisodef} and there is a semiorthogonal decomposition
\begin{equation*}
    \D(X^{n,k})=\langle \D(\bar{\Q}), \O(1), \O(2), \dots, \O(n+2-2k) \rangle
\end{equation*}
where $\bar{\Q}$ is the hyperbolic reduction of $p^{n,k}$ with respect to the smooth $m$-section constructed in Definition \ref{hypreddef}.
\end{prop}
\begin{proof}
Choose some hyperplane $\P(W_{m-1})\subset \P^m:=\P(W_m)$. Let $p'\colon \Q'\to \P^{k-1}$ be the hyperbolic reduction of $p^{n,k}$ with respect to the smooth $(m-1)$-section $\P(W_{m-1})\times \P^{k-1}\subset \Q^{n,k}$. Then $p'$ is a flat quadric surface bundle with a smooth section given by the projectivization of $(W_m/W_{m-1})\otimes \O_{\P^{k-1}}$. Proposition 1.1 (3) in \cite{kuzqbhe} deduces that
\begin{equation*}
    \D(\P^{k-1}, \B_0^{n,k})\cong \D(\P^{k-1}, \B_0')
\end{equation*}
where $\B_0'$ is the even Clifford algebra of $p'$. Then $\D(\P^{k-1}, \B_0')\cong \D(\bar{\Q})$ by Theorem \ref{main1} or \ref{main2} and we obtain the semiorthogonal decomposition from the SOD in (\ref{cisod}).
\end{proof}
\begin{remark}
Note that Proposition 1.1 (3) in \cite{kuzqbhe} only applies to flat quadric bundles and $\bar{\Q}\to \P^{k-1}$ is not flat when $p^{n,k}\colon \Q^{n,k}\to \P^{k-1}$ has fibers of corank $2$. This is why the proof has to go through the middle hyperbolic reduction $p'\colon \Q'\to \P^{k-1}$.
\end{remark}

For smooth complete intersections of three quadrics, we have better results.

\begin{prop} \label{ci3q}
Assume $\kk$ is algebraically closed with $\cchar(\kk)=0$. Let $Y^{2m}$ be the smooth complete intersection of three quadrics in $\P^{2m+3}$. Assume that $Y^{2m}$ contains a $\P^{m-1}$ for $m\geqslant 6$, which is automatically satisfied for $1\leqslant m \leqslant 5$. Then we have a semiorthogonal decomposition
\begin{equation*}
    \D(Y^{2m}) =\langle \D(S^{2m}, \A^{2m}), \O_{Y^{2m}}(1), \O_{Y^{2m}}(2), \dots, \O_{Y^{2m}}(2m-2) \rangle
\end{equation*}
where $S^{2m}$ is the resolution of the double cover over $\P^2$ ramified along a nodal curve of degree $2m+4$ and $\A^{2m}$ is an Azumaya algebra on $S^{2m}$. Moreover, $Y^{2m}$ for $m\geqslant 3$ is rational and $Y^4$ is rational when $[\A^{4}]\in \Br(S^4)$ is trivial. 
\end{prop}
\begin{proof}
Firstly, we claim that $Y^{2m}$ contains a $\P^{m-1}$ for $1\leqslant m \leqslant 5$. Let $F_{m-1}$ be the Hilbert scheme of $\P^{m-1}$'s on a quadric of dimension $2m+2$. It is the zero locus of a section in 
\begin{equation*}
    \Gamma(\Gr(m, 2m+4), \Sym^2\R_m)
\end{equation*}
where $\R_m$ is the universal subbundle on $\Gr(m, 2m+4)$. Thus, $F_{m-1}\subset \Gr(m, 2m+4)$ has codimension at most $m(m+1)/2$. Since $3m(m+1)/2\leqslant \dim \Gr(m, 2m+4)$ when $1\leqslant m \leqslant 5$, $Y^{2m}$ contains a $\P^{m-1}$ in these cases.

In previous notations, $Y^{2m} = X^{2m+2, 3}$. Let $p'\colon \Q'\to \P^{k-1}$ be the hyperbolic reduction of $p^{2m+2,3}\colon \Q^{2m+2,3}\to \P^2$ with respect to the smooth $(m-1)$-section $\P^{m-1}\times \P^2\subset \Q^{n,k}$. Then $p'$ is a flat quadric surface bundle. Proposition 1.1 (3) in \cite{kuzqbhe} deduces the Morita equivalence 
\begin{equation*}
    \D(\P^2, \B_0^{2m+2,3})\cong \D(\P^2, \B_0')
\end{equation*}
where $\B_0'$ is the even Clifford algebra of $p'$. From the smoothness of $Y^{2m}$, we get that $\Q^{n,k}$ and $\Q'$ are also smooth. In addition, the (first) degeneration locus of $q^{2m+2,3}\colon \O_{\P^2}^{2m+4}\to \O_{\P^2}(1)$ is a curve of degree $2m+4$ and the degeneration locus is preserved under hyperbolic reduction. By Lemma \ref{bealem}, the curve is nodal along the locus where fibers of $p'$ have corank $2$. From Theorem \ref{main3}, we get that
\begin{equation*}
    \D(\P^2, \B_0')\cong \D(S^{2m}, \A^{2m})
\end{equation*}
and $[\A^{2m}]\in \Br(S^{2m})$ is trivial if and only if $p'$ has a rational section. We get the semiorthogonal decomposition from the SOD in (\ref{cisod}).

Lastly, by Example 1.4.4 in \cite{beaprym}, we have that $Y^{2m}, m\geqslant 2$ is birational to the hyperbolic reduction $\Q_l$ of $p^{2m+2,3}$ with respect to a smooth $1$-section $l\times \P^2$ where $l \cong \P^1$ is a line on $Y^{2m}$. It can be computed similarly that the Hilbert scheme of planes on a quadric of dimension $2m+2$ has codimension at most $6$ in $\Gr(3,2m+4)$. We get that $\Q_l\to \P^2$ has a section for $m\geqslant 3$. Hence, $Y^{2m}$ is rational for $m\geqslant 3$.  Moreover, if $m=2$ and $[\A^4]=0$, then $\Q_l\cong \Q'\to \P^2$ has a rational section and thus $Y^4$ is rational. 
\end{proof}
\begin{remark}
We expect that Conjecture \ref{conj} holds for every flat quadric bundle of relative even dimension under the same hypotheses. If it is true, then we do not need to take hyperbolic reduction so that the results of quadric surface bundles can be applied. In this case, the assumption that $Y^{2m}$ contains a $\P^{m-1}$ can be removed for getting the semiorthogonal decomposition.
\end{remark}
\appendix
\section{Non-commutative schemes}\label{ncschsec}
In the appendix, we give an overview of derived categories of non-commutative schemes as in Definition \ref{ncdef} generalizing Appendix D in \cite{kuzhs} and \cite[\S2.2]{xienodaldp5}. We will discuss the relations among $\DD_{\QCoh}, \DD(\QCoh)$, $\DD(\Coh)$ and prove the projection formula (Proposition \ref{projf}).

\begin{defn}\label{ncdef}
A pair $(X,\A_X)$ is a \textit{non-commutative scheme} if $X$ is a noetherian scheme, $\A_X$ is a sheaf of $\O_X$-algebras and a quasi-coherent $\O_X$-module. A morphism 
\begin{equation*}
    H=(h, h_{\A})\colon (X,\A_X)\to (Y, \A_Y)
\end{equation*}
of non-commutative schemes consists of a morphism $h\colon X\to Y$ of schemes and a homomorphism $h_{\A}\colon h^*\A_Y\to \A_X$ of $\O_X$-algebras.
\end{defn}
Denote by 
\begin{itemize}
    \item $\A_X^{\op}$ the opposite algebra of $\A_X$,
    \item $\Mod(X, \A_X)$ the category of right $\A_X$-modules,
    \item $\QCoh(X,\A_X)$ the category of quasi-coherent sheaves on $X$ with right $\A_X$-module structures,
    \item $\Coh(X,\A_X)$ the category of coherent sheaves on $X$ with right $\A_X$-module structures.
\end{itemize}

Further denote by 
\begin{itemize}
    \item $\DD, \DD^-, \D$ the unbounded, bounded above and bounded derived categories,
    \item $\DD^*(X,\A_X)$ the derived category $\DD^*(\Coh(X,\A_X))$ for $*=\emptyset, -, \textrm{b}$,
    \item $\DD_{\QCoh}(X, \A_X)$ (resp., $\DD_{\Coh}(X, \A_X)$) the unbounded derived category of right $\A_X$-modules with quasi-coherent (resp., coherent) cohomologies.
\end{itemize}

There are pairs of adjoint functors
\begin{equation}\label{adfunctor}
    \begin{tikzcd}[column sep =large]
    \QCoh(X) \arrow[shift left=0.5ex]{r}{-\otimes \A_X} & \QCoh(X,\A_X) \arrow[shift left=0.5ex]{l}{j_Q},
    \end{tikzcd}
    \begin{tikzcd}[column sep =large]
    \Mod(X) \arrow[shift left=0.5ex]{r}{-\otimes \A_X} & \Mod(X,\A_X) \arrow[shift left=0.5ex]{l}{j_M}
    \end{tikzcd}
\end{equation}
where $j_Q, j_M$ are forgetful functors, and $-\otimes \A_X$ is left adjoint to $j_Q$ and $j_M$. When $\A_X$ is a coherent sheaf, there is an additional pair of adjoint functors
\begin{equation*}
    \begin{tikzcd}[column sep =large]
    \Coh(X) \arrow[shift left=0.5ex]{r}{-\otimes \A_X} & \Coh(X,\A_X) \arrow[shift left=0.5ex]{l}{j}
    \end{tikzcd}
\end{equation*}
where $j\colon \Coh(X, \A_X) \to \Coh(X)$ is the forgetful functor.

Recall that the \textit{coherator} of $X$ is the functor $Q_X$ right adjoint to the inclusion $\QCoh(X)\hookrightarrow \Mod(X)$. For example, if $X$ is affine, then $Q_X(\F)$ for $\F\in\Mod(X)$ is the quasi-coherent sheaf $\wt{\Gamma(X, \F)}$ associated with $\Gamma(X, \F)$. Note that $\wt{\Gamma(X, \F)}\in \QCoh(X,\A_X)$ if $\F\in \Mod(X,\A_X)$. Thus, $Q_X$ induces a coherator functor
\begin{equation}
    Q_{\A_X}\colon \Mod(X, \A_X) \to \QCoh(X,\A_X),
\end{equation}
which is right adjoint to the inclusion $\QCoh(X,\A_X)\hookrightarrow \Mod(X,\A_X)$.

Let $T$ be an abelian category and let $K(T)$ be its homotopy category. Recall that a complex $I^\bullet$ of objects in $T$ is called \textit{K-injective} if $\Hom_{K(T)}(M^\bullet, I^\bullet)=0$ for every acyclic complex $M^\bullet$. In particular, a bounded below complex of injectives is K-injective. 

For a scheme $X$ and a complex $\cK^\bullet$ of $\O_X$-modules, $\cK^\bullet$ is called \textit{K-flat} if the complex
\begin{equation*}
    \Tot(\F^\bullet\otimes_{\O_X} \cK^\bullet)
\end{equation*}
is acyclic for every acyclic complex $\F^\bullet$ of $\O_X$-modules. In particular, a bounded above complex of flat $\O_X$-modules is K-flat. We can define a similar notion for a non-commutative scheme $(X,\A_X)$. We say that a right $\A_X$-module $\cK$ is \textit{right flat} if $\cK \otimes_{\A_X} -$ is an exact functor on $\Mod(X, \A_X^{\op})$ and a complex $\cK^\bullet$ of right $\A_X$-modules is \textit{right K-flat} if the complex
\begin{equation*}
    \Tot(\cK^\bullet\otimes_{\A_X} \F^\bullet)
\end{equation*}
is acyclic for every acyclic complex $\F^\bullet$ of left $\A_X$-modules. As with before, a bounded above complex of right flat $\A_X$-modules is right K-flat. Replacing $\A_X$ by $\A_X^{\op}$, we get notions for left flat and left K-flat.

\begin{lemma}\label{injflat}
(1) $\Mod(X,\A_X)$, $\QCoh(X,\A_X)$ are Grothendieck abelian categories and every complex in $\Mod(X, \A_X)$ or $\QCoh(X,\A_X)$ has a K-injective resolution.

(2) For every complex $\G^\bullet$ of right $\A_X$-modules, there exists a right K-flat complex $\cK^\bullet$ whose terms are right flat $\A_X$-modules and a quasi-isomorphism $\cK^\bullet \to \G^\bullet$ which is termwise surjective. The same is true for complexes of left $\A_X$-modules.
\end{lemma}
\begin{proof}
(1) $\QCoh(X)$ is a Grothendieck abelian category. The abelian category structure, direct sums and exact filtered colimits on $\QCoh(X)$ carry over to $\QCoh(X,\A_X)$. By the adjointness (\ref{adfunctor}), $\QCoh(X,\A_X)$ has a generator $U\otimes \A_X$ where $U$ is a generator of $\QCoh(X)$. Hence, $\QCoh(X,\A_X)$ is a Grothendieck abelian category. The existence of K-injective complexes follows from \cite[Tag 079P]{stacks-project}. The proof for $\Mod(X,\A_X)$ is similar.

(2) We only need to prove for complexes of right $\A_X$-modules. The proof is a modification of \cite[Tag 06YF]{stacks-project} and it suffices to show that $\Mod(X,\A_X)$ has enough right flat objects. We know that $\Mod(X)$ has enough flat objects. For $\G\in \Mod(X,\A_X)$, there is a surjection $\F\twoheadrightarrow j_M(\G)$ from a flat $\O_X$-module $\F$. Then its adjoint map $\F\otimes \A_X\to \G$ is also a surjection and $\F\otimes \A_X$ is a right flat $\A_X$-module. Now $\cK^\bullet$ can be constructed in the same way as \textit{loc. cit.}, which is the filtered colimit of a nice sequence of bounded above complexes of right flat $\A_X$-modules.
\end{proof}

Thanks to the lemma above, it makes sense to talk about the right adjoint functor $RQ_{\A_X}$.

\begin{lemma}\label{qceqlem}
The natural functor $\DD(\QCoh(X,\A_X))\to \DD_{\QCoh}(X,\A_X)$ is an equivalence with quasi-inverse given by $RQ_{\A_X}$.
\end{lemma}
\begin{proof}
Since the coherator functors $Q_{\A_X}\colon \Mod(X, \A_X)\to \QCoh(X,\A_X)$ and $Q_X\colon \Mod(X)\to \QCoh(X)$ commute with forgetful functors, the claim is a consequence of \cite[Tag 09T4]{stacks-project}.
\end{proof}

\begin{lemma}\label{coheqlem}
When $\A_X$ is coherent, the natural functors 
\begin{equation*}
    \DD^*(X,\A_X)\to \DD_{\Coh}^*(\QCoh(X,\A_X)) \to \DD_{\Coh}^*(X,\A_X)  
\end{equation*}
for $*=-, \textrm{b}$ are equivalences.
\end{lemma}
\begin{proof}
The equivalence of the second functor follows from Lemma \ref{qceqlem}. For the first functor, we will modify the proof in \cite[Tag 0FDA]{stacks-project}. We claim that if there is a surjection $\G \twoheadrightarrow \F$ for $\F\in \Coh(X,\A_X)$ and $\G\in \QCoh(X,\A_X)$, then there is some $\G'\in \Coh(X,\A_X)$ that surjects onto $\F$. Consequently, the first functor is an equivalence by \cite[Tag 0FCL]{stacks-project}. We know $j_Q(\G)$ is a filtered union of coherent submodules $\G_i$. Let $\G_i'= \IIm(\G_i\otimes \A_X\to \G)$ be the image of the map adjoint to the inclusion $\G_i\hookrightarrow j_Q(\G)$. Then $\G$ is the filtered union of coherent $\A_X$-submodules $\G_i'$ and one of them will be $\G'$. 
\end{proof}

Given a morphism $H=(h, h_{\A})\colon (X,\A_X)\to (Y, \A_Y)$, a \textit{push-forward} functor is defined by
\begin{equation*}
    H_* \colon  \QCoh(X,\A_X) \to  \QCoh(Y,\A_Y)
\end{equation*}
where as a quasi-coherent sheaf, $H_*\F$ is given by $h_*\F$ for $\F\in \QCoh(X,\A_X)$, and a right $\A_Y$-module structure on $h_*\F$ is induced by
\begin{equation*}
    (h_*\F)\otimes \A_Y \cong h_*(\F\otimes h^*\A_Y) \xrightarrow{h_{\A}} h_*(\F\otimes \A_X)\to h_*\F.
\end{equation*}
A \textit{pull-back} functor is defined by
\begin{equation*}
    H^*\colon \QCoh(Y,\A_Y) \to \QCoh(X,\A_X)
\end{equation*}
where for $\G\in \QCoh(Y,\A_Y)$,
\begin{equation*}
    H^*\G:=(h^*\G)\otimes_{h^*\A_Y} \A_X \cong (h^{-1}\G)\otimes_{h^{-1}\A_Y} \A_X.
\end{equation*}

In the lemma below, we keep the same notations $RH_*$, $LH^*$ for the derived push-forward and pull-back functors induced from the original ones, respectively.
\begin{lemma}\label{derivedfun}
Let $H=(h, h_{\A})\colon (X,\A_X)\to (Y, \A_Y)$ be a morphism between non-commutative schemes.

(1) There exists a right derived functor
\begin{equation*}
    RH_* \colon \DD(\QCoh(X,\A_X)) \to \DD(\QCoh(Y,\A_Y)).
\end{equation*}
When $h$ is proper and $\A_Y$ is coherent, it induces
\begin{equation*}
    RH_* \colon \DD^*(X,\A_X) \to \DD^*(Y,\A_Y)
\end{equation*}
for $*=-, \textrm{b}$.

(2) There exists a left derived functor
\begin{equation*}
    LH^*\colon \DD(\QCoh(Y,\A_Y)) \to \DD(\QCoh(X,\A_X)).
\end{equation*}
When $\A_X$ is coherent, it induces
\begin{equation*}
    LH^*\colon \DD^-(Y,\A_Y) \to \DD^-(X,\A_X).
\end{equation*}

(3) $H^* \dashv H_*$ and $LH^*\dashv RH_*$ are adjoint functors.
\end{lemma}
\begin{proof}
(1) By Lemma \ref{injflat} (1), the K-injective resolutions exist for $\DD(\QCoh(X,\A_X))$ and thus the right derived functor $RH_*$ can be defined. When $h$ is proper, we have an induced functor
\begin{equation*}
    RH_* \colon \DD_{\Coh}(\QCoh(X,\A_X)) \to \DD_{\Coh}(\QCoh(Y,\A_Y)).
\end{equation*}
When $\A_Y$ is coherent, from Lemma \ref{coheqlem}, we get the right derived functor
\begin{equation*}
    RH_* \colon \DD^*(X,\A_X) \to \DD_{\Coh}^*(\QCoh(X,\A_X)) \to \DD_{\Coh}^*(\QCoh(Y,\A_Y))\cong \DD^*(Y,\A_Y).
\end{equation*}
for $*=-, \textrm{b}$.

(2) Given $\G^\bullet \in \DD(\QCoh(Y,\A_Y))$, we define
\begin{equation*}
    LH^*\G^\bullet:= RQ_{\A_X}(H^*\cK^\bullet) \in \DD(\QCoh(X,\A_X))
\end{equation*}
where $\cK^\bullet$ is the right K-flat resolution of $\G^\bullet$ constructed in Lemma \ref{injflat} (2) and $RQ_{\A_X}$ is the derived coherator in Lemma \ref{qceqlem}. Standard arguments show that this is well-defined.

Given $\G\in \Coh(Y, \A_Y)$, we claim that $H^*\G\in\Coh(X,\A_X)$ when $\A_X$ is coherent. This is a local question. Assume that $(X,\A_X) \cong (\Spec A, \wt{R_A})$, $(Y,\A_Y) \cong (\Spec B, \wt{R_B})$ and $\G \cong \wt{M}$. There is a surjection $B^n\twoheadrightarrow M$ for some $n$ and it induces surjections
\begin{equation*}
    R_B^n\twoheadrightarrow M, \quad R_A^n\twoheadrightarrow M\otimes_{R_B} R_A.
\end{equation*}
Then $M\otimes_{R_B} R_A$ is a finitely generated $A$-module because $R_A$ is such. 

Given $\G^\bullet \in \DD^-(Y,\A_Y)$, we have
\begin{equation*}
    LH^*\G^\bullet \in \DD_{\Coh}^-(\QCoh(X,\A_X)) \cong \DD^-(X,\A_X),
\end{equation*}
where the equivalence is given by Lemma \ref{coheqlem}.

(3) It suffices to prove the adjointness for $H^*\dashv H_*$. For $\F\in \Mod(X,\A_X)$ and $\G\in\Mod(Y,\A_Y)$, we have
\begin{align*}
    \Hom_{\A_X}(H^*\G,\F) & \cong\Hom_{\A_X}(h^*\G\otimes_{h^*\A_Y}\A_X, \F) \\
    & \cong \Hom_{h^*\A_Y}(h^*\G, \F)\\
    & \cong \Hom_{\A_Y}(\G, H_*\F).
\end{align*}
\end{proof}

\begin{prop}[Projection formula] \label{projf}
Let $H=(h, h_{\A})\colon (X,\A_X)\to (Y, \A_Y)$ be a morphism of non-commutative schemes in Definition \ref{ncdef}. 

(1) Given $\F\in \DD(\QCoh(X,\A_X^{\op}))$ and $\G\in \DD(\QCoh(Y,\A_Y))$, there is a natural map 
\begin{equation}\label{projmap}
    \G \otimes_{\A_Y}^{\LL} RH_*(\F) \to RH_*(LH^*(\G) \otimes_{\A_X}^{\LL} \F),
\end{equation}
and it is an isomorphism in $\DD(\QCoh(Y))$.

(2) Assume that $h:X\to Y$ is proper and $\A_X, \A_Y$ are coherent. Given $\F\in \DD^-(X,\A_X^{\op})$ and $\G\in \DD^-(Y,\A_Y)$, the natural map (\ref{projmap}) is an isomorphism in $\DD^-(Y)$.
\end{prop}
\begin{proof}
The derived functors involved are defined in Lemma \ref{derivedfun} and the natural map (\ref{projmap}) is induced by the adjointness $LH^*\dashv RH_*$. The proof of the proposition is the same as the proof of Lemma 2.5 in \cite{xienodaldp5}.
\end{proof}

\bibliography{qsb.bib}
\bibliographystyle{abbrv}
\end{document}